\documentclass[11pt,a4paper]{article}
\usepackage[latin1]{inputenc}
\usepackage{indentfirst}
\usepackage{amsmath}
\usepackage{amsfonts}
\usepackage{amsthm}
\usepackage{amssymb}
\usepackage{graphics}
\usepackage{graphicx}
\usepackage{multicol}

\def\B{\mathcal{B}}

\def\F{\mathcal{F}}
\def\MM{\mathcal{M}}
\def\N{\mathbb{N}}

\def\P{\mathbb{P}}
\def\E{\mathbb{E}}

\def\EE{\mathcal{E}}

\def\R{\mathbb{R}}

\def\t{\textrm}
\def\w{\widetilde}
\def\M{\mathcal{M}}

\def\ZZ{\mathbb{Z}}

\def\T{\mathbb{T}}
\def\XX{\mathcal{X}}

\def\ind{{\mathchoice {\rm 1\mskip-4mu l} {\rm 1\mskip-4mu l}
{\rm 1\mskip-4.5mu l} {\rm 1\mskip-5mu l}}}

\newcommand {\G}{\mathbb{G}}

\newcommand{\be} {\begin{equation}}
\newcommand{\ee} {\end{equation}}
\newcommand{\bea} {\begin{eqnarray}}
\newcommand{\eea} {\end{eqnarray}}
\newcommand{\Bea} {\begin{eqnarray*}}
\newcommand{\Eea} {\end{eqnarray*}}

\newtheorem{Thm}{Theorem}
\newtheorem*{Thm*}{Theorem}
\newtheorem{Lem}{Lemma}

\newtheorem{Cor}{Corollary}

\newtheorem{Def}{Definition}

\theoremstyle{definition} \newtheorem{Ass}{Assumption}
\theoremstyle{definition} \newtheorem*{AssA}{Assumption A}
\theoremstyle{definition} \newtheorem*{key}{Key words}
\theoremstyle{definition} \newtheorem*{ams}{A.M.S. Classification}
\theoremstyle{remark}
\theoremstyle{remark}
\theoremstyle{remark}

\begin{document}
\title{Ancestral lineages and limit theorems    \\ for branching Markov chains}
\author{Vincent Bansaye \thanks{CMAP, Ecole Polytechnique, CNRS, route de
    Saclay, 91128 Palaiseau Cedex-France; E-mail: \texttt{vincent.bansaye@polytechnique.edu}}}
\maketitle \vspace{3cm}
\begin{abstract}
We consider a multitype branching model in discrete time,  where the type of each  individual is a trait, which belongs to some general state space. 
Both the reproduction  law and the trait inherited by the offsprings  may depend on the trait of the mother and the environment.
We study the long time behavior of the population and the ancestral lineage of typical individuals under general assumptions. 
We focus on  the  growth rate, the trait distribution among the population for large time, so as  local densities and the position of extremal individuals.  
A key role is played  by well chosen (possibly non-homogeneous) Markov chains. It relies in particular on an extension of many-to-one formula    \cite{guyon, BDMT} and spine decomposition  in the vein of \cite{LPP, KLPP, GeBa}. The applications use properties of the underlying  genealogy  and sufficient conditions for the ergodic convergence of Markov chains.
\end{abstract}
\begin{key} Branching processes, Markov chains, Varying environment, Genealogies.
\end{key}
\begin{ams}60J80, 60J05, 60F05, 60F10
\end{ams}


\section{Introduction}
We are interested in a branching Markov chain, which means a multitype branching process whose number of types may be infinite. The environment may evolve (randomly)
but  when the environment,  each individual evolves independently and the (quenched branching property hold).  \\

Let $(E,T)$ be a pair consisting of a set $E$ of environments 
and an invertible map $T$ on $E$. One can keep in mind the case when the  environment is ${\bf e}=(e_i:  i \in \ZZ)$ and $T{\bf e}=(e_{i+1} : i \in \ZZ)$. \\
Let $(\XX, \B_{\XX})$ be a measurable space which gives the state space of the branching Markov chain. The example $\XX \subset\R^d$ will be relevant for the applications.\\
For each $k\in \N$ and ${\bf e}\in E$, let $P^{(k)}(.,{\bf e},.)$ be a function from $\XX \times \B_{\XX^k}$ into $[0,1]$
which satisfies \\
a) For each $x\in \XX$, $P^{(k)}(x, {\bf e}, .)$ is  a probability measure on $(\XX, \B_{\XX^k})$. \\
b) For each $A\in \B_{\XX^k}$, $P^{(k)}(.,{\bf e},A)$ is a  $\B_{\XX}$ measurable function. \\ \\
In the whole paper, we use the classical notation $u=u_1 u_2...u_n$ with $u_i\in \N^*$ to identify each individual in the population. We denote by $\vert u \vert =n$ the generation of the individual $u$,
by  $N(u)$ the number of offsprings of the individual $u$ and by $X(u) \in  \XX$ the trait (or position) of the individual  $u$. \\ 

For any generation, each individual with trait  $x \in \XX$ which lives in environment ${\bf e} \in E$ gives birth independently to a random number of offsprings, whose 
law both depend on $x$ and ${\bf e}$.  This number of offsprings is distributed as a r.v. $N(x,{ \bf e})$ whose mean is denoted by
$$m(x,{\bf e})=\E(N(x,{\bf e})). $$
In the whole paper, we assume that $m(x, {\bf e})>0$ for each $x\in \XX, {\bf e} \in E$ for convenience. A natural framework for our models will be given  ${\bf e}=(e_i:  i \in \ZZ)$ and $N(x,{\bf e})$ depends only on $x$ and $e_0$, so that
$e_i$ yields the environment in generation $i$   and the reproduction law in generation $i$ just depends on $e_i$. \\
If the environment is ${\bf e}$, we denote by $\P_{{\bf e}}$ the associated probability. 
The distribution of the traits of the offsprings of 
the individual $u$ living in generation $n$ ($\vert u\vert=n)$ is given by  
\Bea
&&\P_{{\bf e}}(X_{u1} \in dx_1, \cdots, X_{uk} \in dx_k \ \vert \ (X(v) :  \vert v\vert \leq n), N(u)=k) \\
&&\qquad \qquad \qquad \qquad \qquad \qquad \qquad \qquad \qquad \qquad=P^{(k)}(X(u),T^n{\bf e},dx_1\cdots, dx_k).
\Eea
In other words, one individual with trait   $x$ living in environment ${\bf e}$ gives birth to a set of individuals $(X_1, \cdots, X_{N(x,{\bf e})})$ whose trait are  specified by 
$(P^{(k)}({\bf e},x,.) : k \in \N , {\bf e}\in  E)$.  \\
This process is  a multitype branching process in varying environment where the types take value in $\XX$. They have been largely studied 
 for finite number of types, whereas much less is known or understood in the infinite case, but some results due to  Seneta, Vere Jones, Moy, Kesten for countable types. \\
The case of branching random walk has  attracted lots of attention from the pioneering works of Biggins. Then  $\XX=\R^d$ and the transitions $P^{(k)}$ are invariant by translation, i.e.
$P^{(k)}(x,{\bf e},x+dx_1\cdots, x+dx_k)$ does not depend on  $x \in \XX$.
Recently, fine results have been obtained about the extremal individuals  and their genealogy for such models, see e.g. \cite{HS, AS}.
Such questions have also been investigated for branching random walk in random environment. In particular the recurrence property  \cite{mull, CoPo}, the survival and the growth rate  \cite{gant, CP, CY},  central limit theorems  \cite{yosh, MN} and large deviations results  \cite{huang2} have been obtained. \\
As far as I see, the methods used for such models and in particular the spectral methods and the martingale arguments are not easily adaptable to the general case we consider. 
We are  motivated by applications to models for biology and ecology such as cell division models for cellular aging \cite{guyon} or parasite infection \cite{ban}  and reproduction-dispersion models in non-homogeneous environment \cite{BL}. Thus, we are here (also) inspired by the utilization of auxiliary Markov chains, branching decomposition and $L^2$ computations, in 
the vein of the works of Athreya and Khang \cite{AKh,AKh2} and Guyon \cite{guyon}. The applications and references will be given along the paper. \\

We are interested here in the evolution of the measure associated to the traits of the individuals:
$$Z_n := \sum_{\vert u\vert=n} \delta_{X(u)}$$
and more specifically by $Z_n(A_n)=\#\{ u  :  \vert u \vert =n, \ X(u)\in A_n\}$. We also define
$$ Z_n(f)=\sum_{\vert u\vert=n} f(X(u)), \quad f_n.Z_n = \sum_{\vert u\vert=n} \delta_{f_n(X(u))}.$$
First, we want to know if the process may survive globally and how it would grow.  Thus, Section 2 yields an expression of the mean growth rate 
of the population relying of the dynamic of the trait and the offspring laws,  in the same vein as  \cite{BL} for metapopulations with a finite number of patches and fixed environment.  Then  (Section 3), we study  the repartition of the population and focus on the asymptotic behavior of the proportions of individuals whose trait belongs to $A$, i.e. $Z_n(\XX)/Z_n(A)$. It is inspired by  
\cite{AKh, guyon, BanHua} and extends the law of large numbers to both varying environment and trait dependent reproduction. 
We add that we take into account some possible renormalization of the traits via a function $f_n$ to cover non recurrent positive cases.  
Finally, in Section 4,  we provide some asymptotic  results about $Z_n(A_n)$,  outside the range of law of large numbers. It relies on the large deviations of the auxiliary process and the trajectory associated with. As an application
we can derive the position of the extremal particles in some monotone models motivated by biology, where new behaviors appear.\\
Let us also mention that the probabilistic approach we follow suggests a way to simulate the long time distribution of the population and will be applied to some biological models motivated by cell division or reproduction-dispersion dynamics.

We end up the introduction with recalling some  classical notations.    If $u=u_1\cdots u_n$ and $v=v_1\cdots v_m$, then $uv=u_1\cdots u_nv_1\cdots v_m$. 
For two different individuals $u, v$ of a tree, write $u<v$ if $u$ is an ancestor of $v$, and denote by $u\wedge v$ the nearest common ancestor of $u$ and $v$ in the means that $|w|\leq |u\wedge v|$ if $w< u$ and $w< v$.

\section{Growth rate of the population}

We  denote by $\rho_{{\bf e}}=\lim_{n\rightarrow \infty} n^{-1}\log \E_{{\bf e}}(Z_n(\XX))$ the  growth rate of the population in the environment ${\bf e}$, when it exists.  \\
We are giving an expression of this growth rate in terms of a Markov chain associated with a random lineage. 
Its transition kernel is defined by
$$P(x,{\bf e},dy):=\frac{1}{m(x,{\bf e})} \sum_{k\geq 1} \P(N(x,{ \bf e})=k) \sum_{i=1}^{k} P^{(k)}(x, {\bf e}, \XX^{i-1}dy\XX^{k-i})$$
so that the  auxiliary Markov chain $X$ is given  by
$$\P_{{\bf e}}(X_{n+1} \in dy \ \vert X_n=x)=P(x,T^n {\bf e},dy).$$
It means that we follow a linage by choosing uniformly at random one of the offsprings at each generation, biased by the number of children.


We assume now that $\XX$ is a locally compact polish space endowed with a complete metric and its Borel $\sigma$ field. Moreover $E$ is a Polish Space and $T$ is an homomorphism. In the rest of the paper, we endow
$\mathcal{M}_1(\XX\times E)$ with the weak topology, where $\mathcal{M}_1(\XX\times E)$ is the space of probabilities on $\XX\times E$. It is the smallest topology such that $\mu \in \mathcal{M}_1(\XX\times E) \rightarrow \int_{\XX\times E} f(z)\mu(dz)$ is continuous as soon as $f$ is continuous and bounded.

\begin{Def}
We say that $X$ satisfies a Large Deviation Principle (LDP)  with good rate function $I_{\bf e}$ in environment ${\bf e}$ when there exists a lower semi continuous function
 $I : \XX\times E \rightarrow \R$ with compact level subsets\footnote{ It means that $\{ \mu \in \mathcal{M}_1(\XX \times E) : I(\mu) \leq l\}$ is compact for the weak topology }
 for the weak topology such that
$$L_n^{{\bf e}}=\frac{1}{n+1} \sum_{k=0}^{n} \delta_{X_k,T^k{\bf e}}$$
satisfies for every $x \in \XX$
$$\limsup_{n\rightarrow \infty} \frac{1}{n} \log \P_{{\bf e}, x}(L_n \in F) \leq -\inf_{z\in F} I_{\bf e}(z)$$
for every closed set $F$ of $\MM_1(\XX\times E)$, and
$$\liminf_{n\rightarrow \infty} \frac{1}{n} \log \P_{{\bf e}, x}(L_n \in O) \geq -\inf_{z\in O} I_{\bf e}(z)$$
for every open set $O$ of $\MM_1(\XX\times E)$.
\end{Def}
\noindent The existence of such a principle is classical for fixed environment $E=\{{\bf e}\}$, finite $\XX$,  under irreducibility assumption. We refer to Sanov's theorem, see e.g. chapter 6.2 in \cite{DZ}. 
We note that the principle can be  extended to periodic environments, taking care of the irreducibility.
Besides, we are using an analogous result for stationary random environment to get forthcoming Corollary \ref{applD}, under Doeblin type conditions, which is due to \cite{timo}.  \\ 

The first question that we tackle now is the mean growth rate of the population.
The branching property yields the linearity of  the operator $\mu \rightarrow m(\mu)=\E_ {{\bf e},\mu}(Z_1(.))$ for some measurable set $A$. \\
In the case of fixed environment, $P$ and $N$ do not depend on ${\bf e}$, so $m$ is also fixed and  the mean growth rate of the process $Z$ is the limit of $\log \parallel m^n \parallel/n$, with $\parallel . \parallel$ an operator norm. 
If $\XX$ is finite, it yields the Perron-Frobenius eigenvalue under strong irreducibility assumption, with a min max representation due to Collatz-Wielandt.  Krein-Rutman theorem gives 
an extension to infinite dimension space requiring compactness of the operator $m$ and strict positivity. \\ In the random environment case,
it corresponds to  the Lyapounov exponent  and quenched asymptotic results can be obtained in the case $\XX$ is finite \cite{Kesten}. Then, the process is a branching process in random environment and we refer to  \cite{AK2, Kaplan} for extinction criteria and \cite{Cohn, Tanny} for its growth rate. \\ 

To go beyond these assumptions and get an interpretation of the growth rate in terms of reproduction-dispersion dynamics, we 
provide here an other characterization. \\
This  is a  functional  large deviation principle  relying on  Varadhan's lemma. It allows to
decouple the reproduction and dispersion in the dynamic.  Thus,  it yields an extension of Theorem 5.3 in \cite{BL} both for varying environment and infinite state pace $\XX$. We refer to this latter article for motivations in ecology, more specifically for   metapopulations. The next Corollary then puts in light the dispersion strategy followed
by  typical individuals  of the population for  large times. 

\begin{Thm} \label{rho}
Assume that $X$ satisfies a LDP with good rate function $I_{{\bf e}}$ in environment ${\bf e}$ and $\log m : \XX \times E\rightarrow (-\infty ,\infty)$ is continuous and bounded. Then, for every $x\in \XX$,
$$\lim_{n\rightarrow \infty} \frac{1}{n}\log \E_{{\bf  e}, \delta_x}(Z_n(\XX)) = \sup_{\mu \in \MM_1(\XX\times E) } \left\{\int_{\XX\times E} \log(m(x,e))\mu(dxde) -I_{{\bf e}}(\mu)\right\}:=\varrho_{{\bf e}}$$
and
$$M_{{\bf e}}:=\bigg\{ \mu \in \MM_1(\XX\times E) :  \int \log( m(x,e))\mu(dxde) -I_{{\bf e}}(\mu)= \varrho_{{\bf e}} \bigg\}$$
is compact and non empty.
\end{Thm}
\noindent  In particular, $\limsup_{n\rightarrow \infty} \frac{1}{n} \log Z_n(\XX) \leq \varrho_{{\bf e}}$ a.s. The limit can hold only on the survival event. It is the case under classical $N\log N$ moment assumption for finite state space $\XX$, see e.g. \cite{LPP}  for one type of individual and fixed environment and \cite{AK2} in random environment. But it is a rather delicate  problem when the number of types is infinite. \\

We introduce now the event
 $$\mathcal S:=\bigg\{\liminf_{n\rightarrow \infty} \frac{1}{n} \log Z_n(\XX)  \ \geq \  \varrho_{{\bf e}}\bigg\}.$$
 Conditionally on  $\mathcal S$, we 
 let $U_n$ be an individual uniformly chosen at random in generation $n$. Let us then focus on  its trait frequency  up to time $n$
 and the associated environment : 
$$\nu_n(A):=\frac{1}{n+1}\#\{ 0\leq i \leq  n  : (X_i(U_n), T^i{\bf e}) \in A\} \qquad (A\in \mathcal B_{\XX\times E}).$$
where $X_i(u)$ is the trait of the ancestor of  $u$ in generation $i$. We prove that the support of  $\nu_n$  converges in probability to $M_{{\bf e}}$ on the event $\mathcal S$.

\begin{Cor} \label{convsupp} Under the assumptions of  Theorem \ref{rho}, we further suppose that $\varrho_{{\bf e}}>0$
and $\mathcal S$ has positive probability.
Then,  for every $x\in \XX$,
$$\P_{{\bf e}, \delta_x}(\nu_n  \in F \vert \mathcal S) \stackrel{n\rightarrow\infty}{\longrightarrow} 0$$
for every closed set $F$ of $\mathcal M_1(\XX\times E)$ which is disjoint of $M_{{\bf e}}$.
\end{Cor}
\noindent This result yields an information on the pedigree \cite{Jagers, NJ} or ancestral lineage of a typical individual. It ensures that the trait  frequency along the lineage of a typical individual
converges to one of the argmax of $\varrho_{{\bf e}}$. We are going a bit farther in the next Section, with a 
description of this  ancestral lineage via size biased random choice, see in particular Lemma \ref{manytoone}.

Let us  now  specify the theorem for stationary ergodic environment ${\bf \mathcal{E}}\in E$, under  Doeblin type assumptions. Following   \cite{timo}, we    let $\pi$ be a $T$ invariant ergodic probability, i.e. $\pi \circ T^{-1}=\pi$ and if $A\in \mathcal B_E$ satisfies
$T^{-1}A=A$, then $\pi(A) \in \{0,1\}$.    Then we need : 
   
\begin{AssA}   There exist a positive integer  $b$, a $T$ invariant subset $E'$ of $E$ and
 a measurable function $M : E \rightarrow [1,\infty)$ such that $\log M \in L^1(\pi)$,   $\pi(E')=1$ and for all $x,y \in \XX$, $A \in \mathcal B_{\XX}$ and ${\bf e} \in E'$,
$$
\label{timoD}
P^b(x, {{\bf e}},A)\leq M({\bf e})P^b(y, {\bf e}, A).
$$
\end{AssA}
\noindent We denote by $V_b(\XX\times E)$  the set  of bounded continuous functions that map $\XX\times E$ into $[1,\infty)$
to state the result.
\begin{Cor} \label{applD}
Under Assumption A, we further suppose that   $\log m(.,\mathcal E)$ is $\pi$ a.s. bounded and continuous. Then   $\pi$  a.s., for every $x\in \XX$,
$$ \lim_{n\rightarrow\infty} \frac{1}{n}\log \E_{{\bf \mathcal{E}}, \delta_x}(Z_n(\XX)) = \sup_{\mu \in \M_1(\XX\times E)} \left\{ \int \log (m(x,e))\mu(dx, de) -I(\mu)\right\},$$
where $I$ is defined by   
$$I(\mu):=\sup \left\{\int_{\XX\times E} \log\left(\frac{u(x,e)}{\int_{\XX}P(x,e, dy) u(y,Te) }\right) \mu(dx,de) : u\in V_b(\XX\times E)\right\}.$$
\end{Cor}


To prove these results, we need the following lemma, where $\mathcal{B}_b(\XX)$ is the set of bounded measurable functions on $\XX$.
\begin{Lem} \label{somed}
Let $F \in \mathcal{B}(\XX^k)$ non-negative. Then, for every $\nu \in \mathcal{M}_1(\XX)$,
$$\E_{{\bf e},\nu} \left(\sum_{\vert  u \vert =n} F(X_0(u),\ldots, X_n(u))\right)=\E_{{\bf e},\nu}\left(F(X_0,\ldots,X_n)\prod_{i=0}^{n-1} m(X_i, T^i{\bf e})\right)$$
where we recall that $X_i(u)$ is the trait of the ancestor of $u$ in generation $n$.
\end{Lem}
\begin{proof}
For every $f_0,\ldots,f_n \in \mathcal{B}(\XX)$ non-negative, by branching property
\Bea
&&\E_{{\bf e},\nu} \left(\sum_{\vert  u \vert =n}f_0(X_0(u))\cdots f_n(X_n(u)) \right) \\
&&\quad=\int \nu(dx_0)f_0(x_0) \int m_1(x_0,{\bf e}, dx_1) \E_{T{\bf e},\delta_{x_1}} 
\left(\sum_{\vert  u \vert =n-1} f_1(X_0(u))\cdots f_{n}(X_{n-1}(u)) \right).\Eea
where 
$$m_1(x_0,{\bf e}, dx_1)=\E_{{\bf e},x_0} \left(\#\{ \vert u \vert =1 : X(u) \in dx_1\}\right)=m(x_0,{\bf e})P(x_0,{\bf e},dx_1).$$
So by induction 
\Bea
&&\E_{{\bf e},\nu} \left(\sum_{\vert  u \vert =n}f_0(X_0(u))\cdots f_n(X_n(u)) \right)\\
&&\qquad =\int_{\XX^n\times A} \nu(dx_0)f(x_0)\prod_{i=0}^{n-1} m(x_i,  T^i{\bf e})P(x_i, T^i{\bf e},dx_{i+1})f(x_{i+1})
\Eea
It completes the proof.
\end{proof}

\begin{proof}[Proof of Theorem \ref{rho}]
The previous lemma applied to $F=1$ ensures that
$$\E_{{\bf e}, \nu}(Z_{n}(\XX))=\E_{{\bf e},\nu}\left(\prod_{i=0}^{n-1} m(X_i, T^i{\bf e})\right).$$
Thus 
$$\E_{{\bf e}, \nu}(Z_{n}(\XX))=\E_{{\bf e},\nu}\left(\exp\left(n\int_{\XX\times E} \log(m(x,e)) L_{n-1}^{{\bf e}}(dx,de)\right)\right) $$
As $\log m$ is bounded and continuous by assumption, so is 
$$\mu \in \mathcal{M}_1(\XX\times E) \rightarrow \phi(\mu)=\int_{\XX\times E} \log (m(x,e))\mu(dx,de).$$
Using the LDP principle satisfied by $L_n^{{\bf e}}$, 
 we can apply  Varadhan's lemma  (see \cite{DZ} Theorem 4.3.1) to the previous function to get the first part of the Theorem.
 

Let us now consider a sequence $\mu_n$ such that
 $$  \int_{\XX\times E} \log(m(x,e))\mu_n(dx de) -I_{{\bf e}}(\mu_n)\stackrel{n\rightarrow \infty}{\longrightarrow} \varrho_{{\bf e}}.$$
Then  $I_e(\mu_n)$ is upper bounded,   which ensures that 
  $\mu_n$ belongs to a sublevel set. By Definition $1$, such a set is compact
so can extract a subsequence $\mu_{n_k}$ which converges weakly in $\mathcal{M}(\XX,E)$.
 As $I_e$ is lower semicontinuous, the limit $\mu$ of this subsequence satisfies
$$\liminf_{k\rightarrow\infty} I_{{\bf e}}(\mu_{n_k}) \geq I_{{\bf e}}(\mu).$$
Recalling that $\phi$ is continuous, we get
$$ \varrho_{{\bf e}}=\lim_{n\rightarrow \infty} \left\{ \int_{\XX\times E} \log( m(x,e))\mu_{\phi(n)}(dx de) -I_{{\bf e}}(\mu_{\phi(n)}) \right\}\leq  \int \log( m(x,e))\mu(dx de) - I_{{\bf e}}(\mu)$$
and  $\mu$ is a maximizer. That ensures that $M_{{\bf e}}$ is compact and non empty.
\end{proof}

\begin{proof}[Proof of Corollary \ref{convsupp}]
 We define for any individual $u$ in generation $n$
$$\nu_n(u)(A)=\frac{1}{n+1}\sum_{0\leq i\leq n} \delta_{X_i(u)}.$$
Using Lemma \ref{somed} with $F(x_0,\ldots,x_n)=1(\frac{1}{n+1}\sum_{0\leq i\leq n} \delta_{x_i} \in F)$, we have
$$\E_{{\bf e}, \nu}\left(\# \{ u  : \vert u\vert =n, \nu_n(u)  \in F\}\right)=\E_{\nu}\left(\exp\left(n\int_{\XX\times E} \log(m(x,e)) L_{n-1}^{{\bf e}}(dx,de)\right)1_{L_n^{{\bf e}} \in F}\right)$$
Applying again  Varadhan's to any   bounded continuous function $\phi :  \mathcal M_1(\XX\times E) \rightarrow \R$ such that, for
every $\mu \in F$,
\be
\label{INEQ}
\phi(\mu)\leq \int_{\XX\times E} \log(m(x,e)) \mu(dx,de)
\ee
we get,
\Bea 
\limsup_{n\rightarrow\infty} \frac{1}{n} \log \E_{{\bf e}, \nu}\left(\# \{ u  : \vert u\vert =n, F_n(u)  \in F\}\right)
&\leq & \sup\{ \phi(\mu)-I_{{\bf e}}(\mu) : \mu \in  \mathcal M_1(\XX\times E)\}. 
\Eea
Let us now check that we can find $\phi$ such that the right hand side is strictly less than $\varrho_{{\bf e}}$. We proceed by contradiction and assume that for every $\phi$ continuous and bounded
which satisfies (\ref{INEQ}), we have $\sup\{ \phi(\mu)-I_{{\bf e}}(\mu) : \mu \in  \mathcal M_1(\XX\times E) \}=\varrho_{{\bf e}}$. Then using the fact $I_{{\bf e}}$ is a good rate function, we obtain that
there exists $\mu(\phi)$ such that   $\phi(\mu(\phi))-I_{{\bf e}}(\mu(\phi))=\varrho_{{\bf e}}$, thanks to the same arguments as the end of the previous proof. Recalling that $\mathcal M_1(\XX\times E)$ can be metrizable by a distance $d$, we define now $\phi_n(\mu):=-nd(\mu,F)+\int_{\XX\times E} \log(m(x,e)) \mu(dx,de)$. We use again the compactness of
 sublevel sets of $I_{{\bf e}}$ to extract a sequence $\mu(\phi_{n_k})$ which converges to $\mu_0$. Then $\mu_0 \in F \cap  M_{{\bf e}}$, which yields the contradiction.

Thus we can choose $\rho'$ such that
$$\limsup_{n\rightarrow\infty} \frac{1}{n} \log \E_{{\bf e}, \nu}\left(\# \{ u  : \vert u\vert =n, F_n(u)  \in F\}\right) <\varrho' < \varrho_{\mathbf e}.$$ 
Adding that
\Bea
\P(\nu_n(U_n) \in F \vert \mathcal S)& \leq & \E \left(\# \{ u  : \vert u\vert =n, \nu_n(u)  \in F\}/Z_n(\XX) \vert S\right) \\
&\leq & e^{-\varrho' n}  \E \left(\# \{ u  : \vert u\vert =n, \nu_n(u)  \in F\}\right) /\P(\mathcal S)
\Eea
for $n$ large enough by definition of $\mathcal{S}$ and that the left hand side goes to $0$ ends up the proof.
\end{proof}

\begin{proof}[Proof of Corollary \ref{applD}]
Under Assumption A,   Theorem 3.3 \cite{timo} ensures that there exists a function $I$ which satisfies   $\pi$ a.s. the  Definition $1$ (uniformly with respect to $x \in \XX$).  The result is then a direct  application of  Theorem \ref{rho}.
\end{proof}

We have given above an expression of the mean growth rate and specified the ancestral lineage of surviving individuals.
It leaves several open questions and we are considering the following ones in the next Section, which are linked :  \\
Does the process grows like its mean when it survives ? \\
How is the population  spread for large times ?

\section{Law of large numbers}

We consider the mean measure under the environment  ${\bf e}$ :
$$m_n(x,{\bf e}, A):=\E_{{\bf e}, \delta_x}\left( Z_n(A)\right)= \mathbb E_{{\bf e}, \delta_x}\left(\#\{ u : \vert u\vert=n,  X(u) \in A\} \right) \qquad (A \in \mathcal B_{\XX}).$$
It yields the mean number of descendant in generation $n$,  whose trait belongs $A$,  of an initial individual with trait $x$. Similarly
we consider its mean number of descendants in generation $n$
We define a new family of  Markov kernel $Q_n$ by 
$$Q_{n}(x,{\bf e}, dy):=  m_1(x, {\bf e},dy)\frac{m_{n-1}(y,T{\bf e},\XX)}{m_n(x,{\bf e},\XX)}.$$
The fact that $Q_{n}(x,{\bf e}, \XX)=1$ for all $n \in \N,x \in \XX,{\bf e} \in E$ comes directly from the branching property.
We introduce  the associated semigroup,  more precisely the successive composition of $Q_j$ between the generations $i$ and $n$ :
$$Q_{i,n}(x,{\bf e}, A)=Q_{n-i}(x,T^i{\bf e}, .)*Q_{n-i-1}(.,T^{i+1}{\bf e}, .)* \cdots * Q_{1}(.,T^{n-1}{\bf e}, .)(A),$$
where we recall the notation $Q(x,.)*Q'(.,.)(A)=\int_{\XX} Q(x,dy)Q'(y,A)$.
The next section links the semigroups $m_n$ and $Q_{0,n}$.

\subsection{The auxiliary process and the many-to-one formula}

The following many-to-one formula links the expectation of the number of individuals whose trait belongs  to  $A$ to the probability that the Markov 
chain associated to  the kernel $Q_{n}$ belongs to $A$.

\begin{Lem} \label{manytoone} For all  $n \in \N,x \in \XX$ and $F \in \mathcal{B}(\XX^{n+1})$ non-negative, we have
$$\E_{{\bf e}, \delta_x}\left(\sum_{\vert u\vert =n}F(X_0(u),\ldots,X_n(u))\right)=m_n(x,{\bf e},\XX) \E_{{\bf e}, x}(F(Y_0^{(n)},\ldots,Y_n^{(n)})),$$
where $(Y_i^{(n)} : i=0,\ldots,n)$ is a non-homogeneous Markov chain with kernels $(Q_{i,n}(.,{\bf e},.) : i=0,\ldots,n-1)$.
In particular for each $f \in \mathcal{B}(\XX)$ non-negative,
 $$m_n(x,{\bf e}, f)=m_n(x,{\bf e},\XX)  Q_{0,n}(x,{\bf e},f),$$
where we recall the notation $\nu (f)=\int_{\XX} f(y) \nu(dy)$.
\end{Lem}
\noindent We note that $m_n(x,{\bf e}, \XX)$ is the mean number of individuals in generation $n$  considered in the previous Section.
Here, 
combining the branching property and the lemma
above yields an other expression of the growth rate:
$$\frac{m_{n+1}(x, {\bf e}, \XX)}{m_n(x,{\bf e}, \XX)}=\int_{\XX} m(y,T^n{\bf e})  Q_{0,n}(x,{\bf e}, dy).$$
The many-to-one formula yields a  spine decomposition of the size-biased tree :  the dynamic of the trait along the spine
follows the non-homogeneous Markov chain $Y$. Going further and describing the whole process seen from the spine requires
additional work. The reproduction of the individuals along the spine is a size biased law and independent process then grow following the original distribution. Such a   decomposition has been firstly  achieved for
Galton-Watson processes in \cite{LPP}. We refer to \cite{KLPP} for an extension to multitype Galton-Watson processes, when the bias among the population relies on the eigenvector of the mean operator, \cite{GeBa} for continuous time and \cite{Geiger} for related 
results in  varying environment. \\
The second part of the Lemma is an extension of
the many one-to-one formula for binary tree \cite{guyon}, Galton-Watson trees \cite{DelMar} and Galton- Watson trees in stationary random environments
\cite{BanHua}.  In continuous time, many-to-one formula and formula for forks  can been found 
in \cite{BDMT}. But these later do not let the reproduction depend on the trait.
We refer to \cite{BC, HR, HRb} for other many-to-one formulas and asymptotic results when reproduction law depend on the trait in some particular cases. \\
For branching random walk in random environment, let us mention the use of induced random walk eliminiating the branching, see e.g. \cite{CoPo}.

\begin{proof} By a telescopic argument : 
$$ \prod_{i=0}^{n-1} Q_{n-i}(x_i,  T^i{\bf e},dx_{i+1})=\frac{m_0(x_{n},{\bf e},\XX)}{m_n(x_0,{\bf e},\XX)} \prod_{i=0}^{n-1} m_1(x_i, T^i{\bf e}, dx_{i+1}).$$
Adding that $m_1(x_i, T^i{\bf e}, dx_{i+1})=m(x_i, T^i{\bf e})P(x_i, T^i{\bf e}, dx_{i+1})$, 
the first part of the lemma is a consequence of Lemma \ref{somed}. 
We then deduce the second part by applying the  identity obtained to $F(x_0,\ldots,x_n)=f(x_n)$.
\end{proof}


Our aim is now to get ride of the expectation and obtain the repartition of the population for large times.
We want to derive it  from the asymptotic distribution of this auxiliary Markov chain with kernel $Q_{n}$ and prove a law of large number
on the proportions of individuals whose trait traited belongs to  $A$.
One approach would be to a find a martingale via maximal eigenvalue and eigenvector, as for finite type and fixed environment. It has been extended 
to branching processes with infinite number of types  in \cite{AthreyaB} but the assumptions required are not easily fulfilled,   at least regarding 
the motivations from biology and ecology we give  in this work.
  Moreover the  generalization to varying environment seems  more adapted to  the technicals described here. 
Thus, we are here inspired by ideas and technics developed in  \cite{AKh,AKh2}
using the branching property  or that in \cite{guyon} relying on $L^2$ computations.

\subsection{Branching decomposition}
In this part,  we focus on the particular case when extinction does not occur and  actually assume that the population has a positive growth rate.
\label{dec}
We have then the following strong law of large numbers, on the following event ensuring geometric growth : 
$$ \mathcal T :=\left\{\forall n, \ Z_n(\XX)>0; \ \liminf_{n\rightarrow\infty} \frac{Z_{n+1}(\XX)}{Z_n(\XX)} >1\right\}.$$

\begin{Thm} \label{brchthm} Let us fix ${\bf e} \in E$ and $\mathcal F$ a bounded subset of  $\mathcal{B}(\XX)$. We assume 
that there exists a measure $Q$ with finite first moment such that for all
$x\in \XX, k,l\geq 0,$
\be
\label{domsto}
\P(N(x,T^k{\bf e}) \geq l)\leq Q[l,\infty).
\ee
Assume also that there exists a sequence of probability measure $\mu_n$ such that
\be
\label{strerg}
\sup_{\lambda \in \mathcal{M}_1(\XX)} \big\vert Q_{i,n}(\lambda,T^i{\bf e}, f\circ f_n)- \mu_n(f) \big\vert \longrightarrow 0,
\ee
uniformly for $n-i\rightarrow \infty$ and $f\in \mathcal F$. Then, 
\be
\label{cvf} 
\frac{f_n.Z_n(f)}{Z_n(\XX)}  - \mu_n(f)\stackrel{n\rightarrow\infty}{\longrightarrow} 0  \qquad \P_{{\bf e}} \text{ a.s. on the event } \mathcal T.
 \ee
\end{Thm}
This result extends \cite{AKh,AKh2} to the case when the reproduction law may depend on the trait and to
to time varying environment. Moreover, here the Markov kernel
$P^{(k)}$ is not forced to be a direct product of the same kernel. It yields a strong law
of large numbers relying on the uniform ergodicity of the auxiliary Markov chain $Q_{i,n}$.  The assumption of   a.s. survival and positive growth rate 
will be relaxed in the next part using $L^2$ assumptions. 

We first state a  result on the sum of independent random variables, which is being used several time.  It is an easy extension of Lemma $1$ in \cite{AKh}, which
itself is proved using  \cite{Kurtz72}.
\begin{Lem} \label{LemmLGN}
Let $\{\mathcal F\}_0^{\infty}$ be a filtration contained in $(\Omega, \mathcal B, \P)$. Let $\{ X_{n,i} : n,i\geq 1\}$ be r.v. such that for each $n$,
 $\mathcal{F}_n$
$\{ X_{n,i} : i\geq 1\}$ are centered independent r.v. Let $\{N_n :  n\geq 1\}$ be non-negative integer valued r.v. such that for each $n$, $N_n$ is $\mathcal F_n$ measurable. \\
We assume that there exists a random measure $Q$ with finite first moment such,
$$\forall t>0, \qquad \sup_{i,n\geq 1} \P( \vert X_{n,i} \vert >t \vert \mathcal F_n)\leq Q(t,\infty) \quad \text{a.s.}$$
Then for each $n_0\geq 1$ and $l>1$,
$$\frac{1}{N_n} \sum_{i=1}^{N_n}  X_{n,i} \stackrel{n\rightarrow \infty}{\longrightarrow} 0$$
a.s. on the event $\{ \forall n \geq n_0 :  N_n>0 ; N_{n+1}/N_n\geq l \}.$
\end{Lem}
\begin{proof} The proof can be simply adapted from the proof of Lemma 1 in \cite{AKh}. For any $\delta>0$ and $l>1$, we define
$$A_{n}:=\left\{\left\vert\frac{1}{N_n} \sum_{i=1}^{N_n}  X_{n,i} \right\vert >\delta; \ \forall k=n_0,\ldots, n :  \frac{N_{k}}{N_{k-1}}\geq l \right\}$$ 
and prove similarly that $\sum_{n\geq n_0} \P(A_n \vert \mathcal F_n)<\infty$.
\end{proof}
We use first this lemma to prove the following result, with
$$A_{n_0,l}:=\{ \forall n \geq n_0 :  N_n>0 ; Z_{n+1}(\XX)/Z_n(\XX) \geq l \}.$$
\begin{Lem} \label{LemLGN} For each $n_0\geq 0$ and $l>1$,
\be
\label{cvmoy}
\frac{1}{Z_{n+p}(\XX)} \sum_{ \vert u\vert=n } m_p(X(u), T^n{\bf e}, \XX)\rightarrow 1 \quad \text{as } n\rightarrow \infty.
\ee
a.s on the event $A_{n_0,l}$.
\end{Lem}
\begin{proof}
The branching property gives a natural decomposition of the population in generation $n+p$, as already used in \cite{AKh2} :
$$Z_{n+p}(\XX)= \sum_{\vert u \vert=n } Z^{(u)}_p(\XX),$$
where $Z^{(u)}$ is the branching Markov chain whose root is the individual $u$ and whose environment is $T^n{\bf e}$.
First, we check that 
  Indeed,
$$Z_{n+p}(\XX)-\sum_{\vert u \vert = n } m_p(X(u),T^n{\bf e}, \XX)= \sum_{\vert u \vert=n} \left[Z^{(u)}_p(\XX)-m_p(X(u),T^n{\bf e},\XX) \right]=Z_n(\XX)\epsilon_{n,p},$$
where 
$$\epsilon_{n,p}:=\frac{1}{Z_{n}(\XX)} \sum_{\vert u\vert=n } X_{p,u}^{(n)}, \qquad X_{p,u}^{(n)}:= Z^{(u)}_p(\XX)-m_p(X(u),T^n{\bf e}, \XX).$$
We note that $(X_{p,u}^{(n)} : \vert u\vert =n)$ are independent conditionally on $\mathcal F_n=\sigma (X(v) : \vert v\vert \leq n)$,  
 $\E(X_{p,u}^{(n)})=0$ and $\vert X_{p,u}^{(n)}\vert \leq  \vert Z^{(u)}_p(\XX)\vert +m_p(X(u),T^n{\bf e}, \XX)$, so that
the stochastic domination assumption (\ref{domsto}) ensures that there exists a measure with finite first moment $Q'$ such that
$$\sup_{u\in \T } \P_{{\bf e}}( \vert X_{p,u}^{(n)} \vert >t \vert  \mathcal{F}_{\vert u\vert})\leq Q'(t,\infty),$$
where we recall that $\T$ is the set of all individuals. 
We can then apply the law of large number of  Lemma \ref{LemLGN} to get that for every $p\geq 0$,
$\epsilon_{n,p}\rightarrow 0$  $\P_{{\bf e}} \text{ a.s.}$ on the event $A_{n_0,l}$, as  $n\rightarrow\infty$. Recalling that $Z_{n+p}(\XX)\geq Z_n(\XX)$ for $n$ large enough, we 
 obtain (\ref{cvmoy}). 
\end{proof}

We can now prove the Theorem.

\begin{proof}[Proof of Theorem \ref{brchthm}]
Using again the branching decomposition,   \bea
 \frac{f_{n+p}.Z_{n+p}(f)}{ Z_{n+p}(\XX)}-\mu_{n+p}(f)  
&=&  \frac{1}{Z_n(\XX)} \sum_{\vert u\vert =n} \frac{Z_n(\XX)}{Z_{n+p}(\XX)} f_{n+p}.Z^{(u)}_p(f)-\mu_{n+p}(f) 
  \quad \text{a.s.} \label{toget}
\eea
and we are proving that the right hand side goes  to $0$ as $n\rightarrow \infty$ on $A_{n_0,l}$, for each $n_0\geq 0, l>1$. For that purpose, we split this expression
 and use the many-to-one formula (Lemma \ref{manytoone})
\Bea
&&\left\vert  \frac{1}{Z_n(\XX)} \sum_{ \vert u \vert =n } \frac{Z_n(\XX)}{f_{n+p}.Z_{n+p}(\XX)} Z^{(u)}_p(f)-\mu_{n+p}(f) \right\vert  \\
&& \qquad \leq \frac{1}{Z_n(\XX)} \left\vert\sum_{ \vert u \vert =n } \frac{Z_n(\XX)}{Z_{n+p}(\XX)} \left[ f_{n+p}.Z^{(u)}_p(f)-m_p(X(u), T{\bf e},f\circ f_{n+p}) \right]\right\vert \\
&&\qquad \qquad +\frac{1}{Z_{n+p}(\XX)} \left\vert \sum_{ \vert u \vert =n } 
m_p(X(u), T^n{\bf e},\XX) [Q_{p}(X(u), T^n{\bf e},f)- \mu_{n+p}(f)] \right\vert \\
&& \qquad  \qquad +\mu_{n+p}(f) \left\vert \sum_{ \vert u \vert =n } \frac{m_p(X(u), T^n{\bf e},\XX)}{Z_{n+p}(\XX)}-1\right\vert.
\Eea
To prove that the first term of this sum goes to zero a.s. on   the event $A_{n_0,l}$ as $n\rightarrow \infty$, we use 
again the law of large numbers of   Lemma \ref{LemLGN}  with now
$$X_{u,n}=\frac{Z_n(\XX)}{f_{n+p}.Z_{n+p}(\XX)}\left[Z^{(u)}_p(f)-m_p(X(u), T{\bf e},f\circ f_{n+p})\right].$$
We note  that 
$$X_{u,n}\leq  f_{n+p}.Z^{(u)}_p(f)+m_p(X(u), T^n{\bf e},f\circ  f_{n+p})\leq M[ Z_p^{(u)}(\XX)+ m_p(X(u), T^n{\bf e}, \XX)],$$ where $M:=\sup_{f \in \mathcal F} 
\parallel f\parallel_{\infty}$. 
Then, 
noting 
$$M_p:=\sup_{n\in \mathbb \N} \left\vert \frac{1}{Z_n(\XX)} \sum_{\vert u\vert =n} \frac{Z_n(\XX)}{Z_{n+p}(\XX)} Z^{(u)}_p(f)-\mu_{n+p}(f) \right\vert
 $$
and recalling (\ref{cvmoy}), we get
\Bea
&&\limsup_{n\rightarrow \infty} \left\vert \frac{1}{Z_n(\XX)} \sum_{u \in \G_n} \frac{Z_n(\XX)}{Z_{n+p}(\XX)} Z^{(u)}_p(f)-\mu_{n+p}(f) \right\vert   \\
&& \qquad \qquad \leq M_p \limsup_{n\rightarrow \infty}  \frac{1}{Z_n(\XX)} \sum_{\vert u\vert=n} \frac{m_p(X(u), T^n{\bf e},\XX)}{Z_{n+p}(\XX)} \\
&&\qquad \qquad  \qquad+M \limsup_{n\rightarrow \infty} \left\vert \sum_{\vert u\vert =n} \frac{m_p(X(u), T^n{\bf e},\XX)}{Z_{n+p}(\XX)}-1\right\vert  \leq M(p).
\Eea
Using now  (\ref{strerg}), we have  $M_p\rightarrow 0$ as $p\rightarrow \infty$, so that $(\ref{toget})$ yields
$$\limsup_{p\rightarrow\infty} \limsup_{n\rightarrow \infty} \left\vert \frac{Z_{n+p}(f)}{ f_{n+p}.Z_{n+p}(\XX)}-\mu_{n+p}(f) \right\vert =0.$$
It ends up the proof.
\end{proof}

\subsection{$L^2$ convergence}
In this section, we state weak and strong law of large numbers by combining $L^2$ computations, the   ergodicity of the auxiliary Markov chain $Y$
and the position of the most recent common ancestor of the individuals. \\
We recall the notations $Q(\lambda,{\bf e}, f)(x)=\int_{\XX^2} \lambda(dx)Q(x,{\bf e},dy) f(y)$ and $\mathcal B(\XX)$ for the set of measurable functions from $\XX$ to $\R$. We note  $\mathcal B_b(\XX)$  the set of measurable functions from $\XX$ 
to $\R$, which are bounded by a same constant $b\geq 0$.

The main  assumption we are using concern the ergodic behavior of the time non-homogeneous auxiliary Markov chain $Y$ associated with the transitions kernels $Q_{i,n}$.

\begin{Ass}
\label{ergo}  Let ${\bf e}_n\in E$, $\mathcal F \subset \mathcal{B}(\XX)$, $f_n \in \mathcal{B}(\XX)$ and $\mu_n \in \mathcal{M}_1(\XX)$ for each $n\in \N$.
 
 (a) For all $\lambda \in \mathcal{M}(\XX)$ and $i\in \N$,
 \be
 \sup_{f \in \mathcal F} \big\vert  Q_{i,n}(\lambda, {\bf e}_n, f\circ f_n)-\mu_n(f) \big\vert \stackrel{n\rightarrow \infty}{\longrightarrow} 0. \nonumber 
\ee

(b) For every $k_n \leq n$ such that $n-k_n\rightarrow \infty$,
\be
 \sup_{\lambda\in \mathcal{M}(\XX), f \in \mathcal F} \big\vert  Q_{k_n,n}(\lambda, {\bf e}_n, f\circ f_n)-\mu_n(f)\big\vert \stackrel{n\rightarrow \infty}{\longrightarrow} 0. \nonumber
\ee
\end{Ass}
\noindent The second assumption (uniform ergodicity) clearly implies the first one.  Sufficient conditions  will be given in the applications. In particular, they are linked to Harris ergodic theorem and more specifically they will be formulated in terms of Doeblin and Lyapounov type  conditions. The function $f_n$ is bound to make the process ergodic if it is not originally. We have for example in mind the case when the auxiliary Markov chain $X_n$ satisfies a central limit theorem, i.e.  $f_n(x)= (x-a_n)/b_n$ and $f(X_n)$ converges to the same distribution whatever the initial value $X_0$ is. Such convergence hold for example for branching random walks. \\

We consider now the genealogy of the population and the time of the most recent common ancestor of two individuals chosen uniformly.
\begin{Ass} \label{sep}
(a) For every $\epsilon>0$, there exists $K \in \N$,  such that for $n$ large enough,
\be
\label{first}
\frac{\E_{{\bf e}_n, \delta _x}(\#\{ u,v : \vert u \vert=\vert v\vert=n, \  u \wedge v \geq  K\})}{m_n(x,{\bf e}_n, \XX)^2} \leq \epsilon.
\ee
Moreover  there exists $C_i \in \mathcal B( \XX^2)$ such that for all  $i\in \N, x,y \in \XX$,
 $$\sup_{n\geq i} \frac{m_{n-i}(y,T^i{\bf e}_n, \XX)}{m_n(x,{\bf e}_n, \XX)} \leq C_i(x,y), \quad \text{with } \E\left( \max\{ C_i(x,X(w)) ^2: \vert w \vert= i+1\}\right) <\infty.$$
(b) For every $K \in \N$,  
\be
\label{second}
\frac{\E_{{\bf e}_n, \delta _x}(\#\{u,v : \vert u \vert=\vert v\vert=n, \ u \wedge v \geq n- K\})}{m_n(x,{\bf e}_n,\XX)^2} \stackrel{n\rightarrow\infty}{\longrightarrow} 0.
\ee
Moreover,
$$\sup_{n\in\N} \E(Z_n(\XX)^2)/m_n(x,{\bf e}_n, \XX)^2 <\infty.$$
\end{Ass}
\noindent These expressions can be rewritten in terms of normalized variance of $Z_n(\XX)$ and more tractable sufficient assumptions can be specified, see the applications. We  also observe  that these assumptions require that $Z_n(\XX)$  has a finite second moment, so  each reproduction law involved in the dynamic has a finite second moment. Moreover  
$m_n(x,{\bf e}_n,\XX)$ has to go to $\infty$.\\
The assumption (\ref{first}) says that the common ancestor is at the beginning of the genealogy. It is the case for Galton-Watson trees, branching processes in random environment and many others ``regular trees''.
The  assumption (\ref{second}) says that  the common ancestor is not at the end  of the genealogy, so   it is weaker. For a simple example where (\ref{first}) is fulfilled but (\ref{second}) is not, one can consider the tree $T_n$ which is composed by a single individual until generation $n-k_n$ and equal to the binary tree between the generations
$n-k_n$ and $n$, with $k_n\rightarrow \infty$.  One can also construct   examples of branching Markov chain with time homogeneous reproduction. As an hint,  we mention the degenerated 
case when the tree is formed by a spine where in each generation, the individual of the spine
has one child outside the spine which gives exactly one child in each generation. More generally, such genealogy may arise   by considering increasing Markov chains 
and increasing mean reproduction (which may be deterministic) with respect to $x\in \XX$.   

\begin{Thm}[Weak LLN] \label{LGNL2} Let ${\bf e}_n \in E^n$, $x\in \XX$, $f_n : \XX \rightarrow \XX$ and $\F \subset \mathcal{B}_b(\XX)$.  

 We assume either that Assumptions 1(a) and   2(a)  hold or that Assumptions 1(b) and 2(b) hold. Then, uniformly  for $f\in \F$,
\be
\label{cvf} 
\frac{f_n.Z_n(f) - \mu_n(f)Z_n(\XX)}{m_n(x,{\bf e}_n, \XX)}\stackrel{n\rightarrow\infty}{\longrightarrow} 0
 \ee
in $L^2_{{\bf e}_n, \delta_x}$ and for all $\epsilon,\eta>0$,
  $$ \P_{{\bf e}_n, \delta_x}\left( \left\vert\frac{f_n.Z_n(f)}{Z_n(\XX)} - \mu_n(f)\right\vert \geq \eta \ ; \ Z_n(\XX)/m_n(x,{\bf e}_n,\XX) \geq \epsilon \right)\stackrel{n\rightarrow\infty}{\longrightarrow} 0.$$
 \end{Thm} 
\noindent We first note that $f_n.Z_n(\ind(A))/ Z_n(\XX)$ is the proportion of individuals in generation $n$  whose trait  belongs to $f_n^{-1}(A)$. The assumptions require either weak ergodicity and
early separation of lineages or strong ergodicity and non-late separation of lineages.\\
We refer to the next section for various applications, in particular we recover in Section \ref{WLLNT} the classical weak law of large numbers for Markov chains along  Galton-Watson trees \cite{DelMar} and   
  branching processes in random environment \cite{BanHua}. \\ 
We also mention that the Theorem  holds also if  $f_n : \XX \rightarrow \XX'$ and can be extended to unbounded $f$ with domination assumptions following \cite{guyon}. Finally, let us mention that the a.s. convergence may fail in the theorem above, even in the field of applications we can have in mind. One can think for example of   an underlying genealogical tree growing very slowly and each individual is attached
   with  i.i.d. random variable, as appears e.g. in tree indexed random walks.\\

\begin{proof} Let us prove the first part of the Theorem under Assumptions 1(a) and  2(a). In the whole proof, 
$x$ is fixed and we omit $\delta_x$ in the notation of the probability and of the expectation. For convenience, we also
write $m(x,{\bf e}_n):= m(x,{\bf e}_n,\XX)$ and  denote
$$g_n(x):=f(f_n(x)) -\mu_n(f).$$
Let us compute for $K\geq 1$, 
\Bea
&&\E_{{\bf e}_n}\left(Z_n(g_n)^2\right)\\
&&\quad = \E_{{\bf e}_n}\left(\sum_{\vert u\vert=\vert v\vert=n } g_n(X(u))g_n(X(v)\right) \\
&&\quad = \E_{{\bf e}_n}\left(\sum_{\substack{\vert u\vert=\vert v\vert=n \\ \vert u\wedge v\vert < K}} g_n(X(u))g_n(X(v))\right)+ 
\E_{{\bf e}_n}\left(\sum_{\substack{\vert u\vert=\vert v\vert=n \\ \vert u\wedge v\vert \geq  K}} g_n(X(u))g_n(X(v))\right)
\Eea
The second term of the right hand side is smaller than
$$2\parallel f\parallel_{\infty}^2\E(\#\{\vert u\vert=\vert v\vert=n : \vert u\wedge v\vert > K\})\leq 2b^2m(x,{\bf e}_n)^2.\epsilon_{K,n},$$
where $\limsup_{n\rightarrow \infty}\epsilon_{K,n}\rightarrow 0$ as $K\rightarrow \infty$ using the first part  of Assumption 2(a). So we just deal with the first term and consider $i=1,\ldots, K$. Thanks  to the branching property,
\Bea
&&\E_{{\bf e}_n}\left(\sum_{\substack{\vert u\vert=\vert v\vert=n \\ \vert u\wedge v\vert =i-1}} g_n(X(u))g_n(X(v))\right) \\
&&\qquad = \E_{{\bf e}_n}\left(\sum_{\substack{\vert w \vert = i-1 \\ \vert wa\vert=\vert wb \vert =i}} \sum_{\substack{\vert u\vert =n  \\ u\geq wa}}\sum_{\substack{\vert v \vert=n  \\ 
v\geq wb}} g_n(X(u))g_n(X(v))\right)\\ 
&&\qquad = \E_{{\bf e}_n}\left(\sum_{\substack{\vert w \vert=i-1\\ \vert wa\vert=\vert wb \vert = i}} R_{i,n}(X(wa))R_{i,n}(X(wb))\right),
\Eea
where the many-to-one formula of Lemma \ref{manytoone} allows us to write
\be
\label{defR}
R_{i,n}(x):=\E_{T^i{\bf e}_{n},\delta_ x}\left(\sum_{\vert u\vert= n-i}g_n(X(u))\right)= m_{n-i}(x,T^i{\bf e}_{n}) Q_{n-i}(x,T^i{\bf e}_{n},g_n).
\ee
Then Assumption 1(a)  ensures that 
$$F_{i,n}(u):=R_{i,n}(X(u))/m_{n-i}(X(u),T^i{\bf e}_n)$$
 goes to $0$ a.s. for each $i\in \N, \vert u \vert=i$ uniformly for $f\in \mathcal F$. We also   note that this quantity is  bounded by $b$. Then,
\Bea
&&m(x,{\bf e}_n,\XX)^{-2}\E_{{\bf e}_n}\left(\sum_{\substack{\vert u\vert=\vert v\vert=n \\ \vert u\wedge v\vert \leq K}} g_n(X(u))g_n(X(v))\right) \\
&&   \quad =\E_{{\bf e}_n}\left(\sum_{\substack{i \leq K, \vert w\vert = i-1  \\ \vert wa\vert=\vert wb \vert= i}} F_{i,n}(wa)F_{i,n}(wb) \frac{m_{n-i}(X(wa),T^i{\bf e}_n)m_{n-i}(X(wb),T^i{\bf e}_n,\XX)}{m_{n}(x,{\bf e}_n)^2}\right). 
\Eea
Adding that 
\Bea
&&F_{i,n}(wa)F_{i,n}(wb) \frac{m_{n-i}(X(wa),T^i{\bf e}_n)m_{n-i}(X(wb),T^i{\bf e}_n)}{m_{n}(x,{\bf e}_n)^2}\\
&& \qquad \qquad \qquad \qquad \leq b^2 \sup_n \frac{m_{n-i}(X(wa),T^i{\bf e}_n)}{m_n(x,{\bf e}_n)} .\sup_n \frac{m_{n-i}(X(wb),T^i{\bf e}_n)}{m_n(x,{\bf e}_n)},
\Eea
the second part of Assumption 2(a) 
ensures the $L^2_{{\bf e}_n}$ convergence (\ref{cvf}), uniformly for $f\in \mathcal F$. \\

The proof of (\ref{cvf}) under Assumptions 1(b) and  2(b) is almost the same, replacing $K$ by $n-k_n$  with $k_n\rightarrow \infty$. Indeed, Assumption 1(b) ensures that there exists $k_n\rightarrow \infty$ such that
$$\frac{\E_{{\bf e}_n, \delta _x}\left(\#\{ \vert u\vert=\vert v \vert=n  : u \wedge v > n- k_n\}\right)}{m_n(x,{\bf e}_n)^2} \stackrel{n\rightarrow\infty}{\longrightarrow} 0,$$
whereas 
\Bea
&&\E_{{\bf e}_n}\left(\sum_{\substack{i \leq n-k_n, \vert w\vert = i-1  \\ \vert wa\vert=\vert wb \vert= i}} F_{i,n}(wa)F_{i,n}(wb) \frac{m_{n-i}(X(wa),T^i{\bf e}_n)m_{n-i}(X(wb),T^i{\bf e}_n)}{m_{n}(x,{\bf e}_n)^2}\right) \\
&&\qquad \qquad \qquad \qquad \qquad \qquad \qquad \qquad\leq \left( \sup_{n-i\geq k_n, x\in \XX} F_{i,n}(x)\right)^2 \frac{\E(Z_n(\XX)^2)}{m_{n}(x,{\bf e}_n)^2}.
\Eea
Assumption 1(b) ensures that $\sup_{n-i\geq k_n, x\in \XX} F_{i,n}(x)\rightarrow 0$ as $k_n\rightarrow \infty$ and the second part of Assumption
2(b) ensures that  $\E_{{\bf e}_n}(Z_n(\XX)^2)/m_{n}(x,{\bf e}_n)$ is bounded. The conclusion is thus the same. 
 
The proof of the last part of the Theorem comes simply from
Cauchy-Schwarz inequality :
\Bea
&&\E_{{\bf e}_n}\left(\ind_{Z_n(\XX)/m(x,{\bf e}_n) \geq \epsilon} \left[\frac{f_n.Z_n(f)}{Z_n(\XX)} - \mu_n(f)\right]\right)^2 \\
&&\leq \E_{{\bf e}_n} \left( \frac{m_n(x,{\bf e}_n)^2}{Z_n(\XX)^2}1_{Z_n(\XX)/m_n(x,{\bf e}_n) \geq \epsilon}\right)\E_{{\bf e}_n} \left(\left[\frac{f_n.Z_n(f)-Z_n(\XX)\mu_n(f)}{m_n(x,{\bf e}_n)} \right]^2\right).
\Eea
The first term of the right-hand side is bounded with respect to $n$. So applying the first part of the theorem to the second term  and using Markov inequality ends up the proof.
\end{proof}

We give now a strong law of large numbers. For that purpose, we define 
$$ V_i(x_0,x_1)=\sup_{ k\geq 0 } \frac{m_{k}(x_0, T^{i}{\bf e}, \XX)m_{k}(x_1, T^{i}{\bf e},\XX)}{m_{i+k}(x,{\bf e},\XX)^2}.$$
\begin{Lem} \label{ctrL2}
Let ${\bf e}\in E$,  $x\in \XX$ and assume that
 \be
\label{Hyp1as}
 \sum_{n\geq 0}m_n(x,{\bf e},\XX)^{-1}<\infty; \quad \sum_{i \geq 1} \E_{{\bf e}, \delta_x}\left(  \sum_{ \substack{  \vert w\vert=i-1 \\  
\vert wa\vert=\vert wb\vert= i}}
 V_i(X(wa),X(wb))\right)<\infty,
 \ee
then $Z_n(\XX)/m_n(x,{\bf e},\XX)$ is bounded in $L^2_{{\bf e}, \delta_x}$.
\end{Lem}
The proof of this Lemma is given after the following main result.
 \begin{Thm}[Strong LLN] \label{coras}  
Let ${\bf e}\in E$,  $x\in \XX$ and  $f\in  \mathcal B_b(\XX)$.\\
 Assume that (\ref{Hyp1as}) hold  and that 
 that there exists     a sequence of probability measure $\mu_n$ on $\XX$  such that
\be
\label{Hyp2as}
\sup_{i \in \N, f\in \F} \sum_{ n\geq i}  \sup_{\lambda\in \mathcal M_1} \big\vert Q_{i,n}(\lambda,T^{i}{\bf e}, f\circ f_n)-\mu_n(f) \big\vert ^2 <\infty.
\ee
Then, 
$$\frac{f_n.Z_n(f) - \mu_n(f)Z_n(\XX)}{m_n(x,{\bf e}_n, \XX)}\stackrel{n\rightarrow\infty}{\longrightarrow} 0 \qquad \P_{{\bf e},\delta_x} \text{ a.s.}
$$
\end{Thm}
$\newline$
\noindent 
The first assumption is related to the genealogy of the population and the second one is linked to the ergodic property 
of  the auxiliary Markov chain $Y$. Both assumptions are stronger that their counterpart of the previous theorem. \\
We  refer to \cite{guyon} for more general conditions on the functions $f\in\F$ in the fixed environment case, when
the   reproduction law does not depend on the position.\\
 We  note that under the Assumptions of the Theorem,   $Z_n(\XX)/m_n(x,{\bf e},\XX)$ is bounded in $L^2_{\bf e}$ thanks to the previous Lemma. It
entails that the probability of the event $\{Z_n(\XX)/m_n(x,{\bf e},\XX) \geq \epsilon\}$ is positive for $\epsilon$ small enough and every $n\geq 1$.
On this event, we get $f_n.Z_n(f)/Z_n(\XX) - \mu_n(f) \rightarrow  0$ a.s. as $n\rightarrow \infty$.

\begin{proof}[Proof of Lemma \ref{ctrL2}]
We again omit  the initial state  $\delta_x$ in the notations  and write  $m_n(x,{\bf e})$ for $m_n(x,{\bf e},\XX)$.
Using the branching property and distinguishing if the common ancestor of two individuals lives before generation $n$ or in generation $n$, we have
\Bea
\E_{\bf e}(Z_n(\XX)^2)&=&\E_{\bf e}\left(\sum_{\vert u\vert=\vert v\vert =n} 1\right) \\
&=& \E_{\bf e}(Z_n(\XX))+\E_{\bf e}\left(\sum_{i\leq n} \sum_{\substack{ \vert w \vert=i-1  \\ \vert wa\vert=\vert wb\vert=i}} 
\sum_{\substack{\vert u \vert=n  : u>wa \\ \vert v \vert=n :  v>wb}} 1\right) \\
&=& m_n(x,{\bf e})+\sum_{i\leq n}  \E_{\bf e}\left( \sum_{\substack{ \vert w \vert=i-1  \\ \vert wa\vert=\vert wb\vert=i}}
m_{n-i}(X(wa),T^{i}{\bf e})m_{n-i}(X(wb),T^{i}{\bf e})\right).
\Eea
Then, 
$$\frac{\E_{\bf e}(Z_n(\XX)^2)}{ m_n(x,{\bf e})^2}\leq \frac{1}{ m_n(x,{\bf e})}+\sum_{i\leq n}  \E_{\bf e}\left( \sum_{\substack{ \vert w \vert=i -1 \\ \vert wa\vert=\vert wb\vert=i}}
V_i(X(wa),X(wb)\right),$$
which ends up the proof.
\end{proof}

\begin{proof}[Proof of Theorem \ref{coras}]
To get the a.s. convergence, we prove that $$\E_{{\bf e}}\left(\sum_{n \geq 1 } \left[\frac{f_n.Z_n(f) - \mu(f)Z_n(\XX)}{m_n(x,{\bf e})}\right]^2\right)<\infty.$$
For that purpose, we use the notations of the proof of the previous Theorem, in particular 
$$g_{n}(x):=f(f_n(x))-\mu_{n}(f)  $$
 and we are inspired by \cite{guyon}. Using Fubini inversion,   the branching property and (\ref{defR}), we have
\Bea
&& \sum_{n\geq 0} m_n(x,{\bf e})^{-2}\E_{{\bf e}}(Z_n(g_n)^2)\\
&&= \E_{{\bf e}}\left(\sum_{n\in \N} \sum_{\vert u\vert=\vert v \vert=n} m_n(x,{\bf e})^{-2}g_n(X(u))g_n(X(v))\right) \\
&&= \E_{{\bf e}}\left(\sum_{n\in \N} \sum_{i \leq n}\sum_{ \substack{\vert u\vert=\vert v\vert=n \\   \vert u\wedge v\vert =i}} m_n(x,{\bf e})^{-2}g_n(X(u))g_n(X(v))\right)\\
&&= \E_{{\bf e}}\left(\sum_{i\leq n} \sum_{\substack{ \vert w \vert=i-1  \\ \vert wa\vert=\vert wb\vert=i}} 
\sum_{\substack{\vert u \vert=n  : u \geq wa \\ \vert v \vert=n :  v \geq wb}} m_n(x,{\bf e})^{-2} g_n(X(u))g_n(X(v))\right)\\
&&\qquad \quad + \E_{{\bf e}}\left(\sum_{n \in \N, \vert u \vert =n } m_n(x,{\bf e})^{-2}g_n(X(u))^2\right)\\
&& \leq \E_{{\bf e}}\left( \sum_{i\leq  n} \sum_{\substack{ \vert w \vert =i -1 \\ \vert wa\vert=\vert wb\vert= i}}  \frac{m_{n-i}(X(wa), T^{i}{\bf e}) m_{n-i}(X(wb),  T^{i}{\bf e})}{m_n(x,{\bf e})^2} 
R_{i,n}(X(wa))R_{i,n}(X(wb))\right) \\
&&\qquad \quad + 2\parallel g_n\parallel_{\infty} \E_{{\bf e}}\left(\sum_{n \in \N}   m_n(x,{\bf e})^{-2}Z_n(\XX)  \right) \\
&&\leq \E_{{\bf e}}\left(\sum_{\substack{i \in \N , \vert w \vert=i-1 \\ \vert wa\vert=\vert wb \vert= i}} V_i(X(wa),X(wb))H_i\right) + b\sum_{n \in \N} m_n(x,{\bf e})^{-1} ,
\Eea
where $b:=(2\parallel f \parallel_{\infty})^2$ and  $$ H_i= \sup_{y,z} \sum_{n\geq i} R_{i,n}(y)R_{i,n}(z), \qquad 
 V_i(x_0,x_1)=\sup_{ n \geq i} \frac{m_{n-i}(x_0, T^{i}{\bf e})m_{n-i}(x_1, T^{i}{\bf e})}{m_{n}(x,{\bf e})^2}.$$
Then, the assumptions ensure that
\Bea
&& \sum_{n\geq 0} m_n(x,{\bf e},\XX)^{-2}\E_{{\bf e}}(Z_n(g_n)^2)\leq b\sum_{n\geq 0}m_n(x,{\bf e})^{-1}\\
&& \qquad \qquad  \qquad \qquad+ \sup_{i\in \N} H_{i}.\sum_{i \in \N} \E\left( \sum_{\substack{ \vert w \vert = i -1\\ \vert wa\vert=\vert wb\vert=i}}
 V_i(X(wa),X(wb))\right)<\infty.
\Eea
Then, $Z_n(g_n)/ m_n(x,{\bf e})\rightarrow 0$ a.s., which completes the proof.
\end{proof}


\subsection{Applications}

We now provide some applications of the previous results. 

\subsubsection{Weak law of large numbers along branching trees}
\label{WLLNT}
First,  we consider the neutral case, which means that the reproduction law 
of the individuals to do not depend on their trait. So the underlying genealogy is simply a branching process  (possibly non-homogeneous) and 
Assumption \ref{sep}(a) can be easily checked. 
We also note that $m_n:=m_n(x,{\bf e},\XX)=\Pi_{j=0}^{n-1}m(x,T^j{\bf e})$ does not depend on $x$.\\
Thanks to Theorem \ref{LGNL2}, we only require weak ergodicity of the auxiliary Markov chain, whose kernel transition simplifies as
$$Q_{n}(x,{\bf e}, dy):=  m_1(x, {\bf e},dy)\frac{1}{m_1(x,{\bf e},\XX)}=P(x,{\bf e},dy).$$
Moreover $W_n=Z_n/m_n$ is (a.s. with respect to the environment) a martingale which converges to a positive limit on the non extinction event thanks to $L^2$ assumptions, so that we obtain
$$ \P_{{\bf e}}\left( \left\vert\frac{f_n.Z_n(f)}{Z_n(\XX)} - \mu_n(f)\right\vert \geq \eta \ ; \forall n\in \N,
  Z_n(\XX)>0 \right)\stackrel{n\rightarrow\infty}{\longrightarrow} 0.$$
  We recover here  classical weak law of large numbers for proportions of individuals with a given trait for  Markov chains along trees, such as  Galton-Watson trees \cite{guyon} (Section 2.2),  \cite{DelMar} (Theorem 1.3) and   
  branching processes in random environment \cite{BanHua} (Theorem 3.2).

\subsubsection{Under Doeblin type conditions}
\label{Doeb}
For convenience, we assume here strong Doeblin type conditions  on the mean measure, in the same vein as  Section 2.
\begin{Ass} \label{annU}
There exist $M: E\rightarrow [1,\infty)$ such that for all  $x\in \XX, {\bf e} \in E$,
$$ m_1(x,{\bf e}, A)\leq M({\bf e})m_1(y,{\bf e},A).$$
\end{Ass}
\noindent We could relax this assumption, for example by requiring such an inequality for 
$m_b$ instead of  $m$, for some $b\geq 1$. We note 
that this assumption  hold if   both  $m(x,{\bf e})$ and $P(x,{\bf e}, .)$ satisfy the analogous condition. We 
 refer to the next part and to \cite{Mukh} for more general conditions in  the-non homogeneous framework.\\
Let us denote
$$\sigma({\bf e}):=\sup_{x\in \XX} \E(N(x,{\bf e})^2), \qquad D({\bf e}):=\frac{\sigma ({\bf e})M({\bf e})  M(T{\bf e}) ^{2}}{m(x,T{\bf e})}$$
to state the result.
\begin{Cor} \label{D}
Let ${\bf e} \in E, x\in \XX$ and $f\in \B_b(\XX)$. We assume that
 Assumption \ref{annU} holds with
\be
\label{coeffD}
\sum_{n\geq 1} \frac{1+D(T^{n-1}{\bf e})}{ m_{n}(x,{\bf e},\XX)}<\infty,  \quad \sum_{n\geq 0}   \prod_{k=0}^n (1-1/M(T^k{\bf e})^2)<\infty.
\ee
Then, $Z_n(\XX)/m_n(x,{\bf e},\XX)$ is bounded in $L^{2}_{{\bf e},\delta_x}$ and 
 $$ \frac{Z_n(f)- Z_n(\XX)Q_{0,n}(x,{\bf e},f)}{m_n(x,{\bf e},\XX)}\stackrel{n\rightarrow\infty}{\longrightarrow} 0 \qquad \P_{{\bf e}, \delta_x} \quad \text{a.s}.$$
 \end{Cor}
\noindent 
The assumptions above ensure that $m_{n}(x,{\bf e},\XX)$ goes to $\infty$ (supercriticality). The assumptions are fulfilled  for example if $m_n$ tends fast enough to $\infty$  and both $\sigma,M, m$ are bounded. \\
The proof here uses Theorem \ref{coras} and some additional lemma allowing to check the assumptions required.  We may derive an application of this result to the random environment framework directly or relax the assumptions to get only convergence in probability. \\
The fact that  $Z_n(\XX)/m_n(x,{\bf e},\XX)$ is bounded  in $L^{2}_{{\bf e},\delta_x}$ ensures that
$Z_n(\XX)\rightarrow \infty$ with positive probability. Using Paley-Sigmund inequality and 
a stronger assumption in the first part of (\ref{coeffD}) to ensure the uniformity with respect to the initial environment, one can prove that 
 $\liminf_{n\rightarrow\infty} Z_n(\XX)/m_n(x,{\bf e},\XX)>0$ on the event  $\{Z_n(\XX)\rightarrow \infty\}$. Then,  $\P_{{\bf e}, \delta_x} \quad \text{a.s}$ on the event $\{Z_n(\XX)\rightarrow \infty\}$,
 we have
$ Z_n(f)/Z_n(\XX) - Q_{0,n}(x,{\bf e},f)\rightarrow 0$ as $n\rightarrow \infty$. 

\begin{Lem} \label{lem} Under Assumption \ref{annU}, for all $x,y \in \XX, {\bf e} \in E, n\geq 0, A \in \mathcal{B}_{\XX}$,
\Bea
&&\frac{m_{n}(x,{\bf e})}{m_{n}(y,{\bf e})} \in [M({\bf e})^{-1}, M({\bf e})], \qquad Q_n(x,{\bf e}, A)\leq M({\bf e})^2Q_n(y,{\bf e}, A). 
\Eea
\end{Lem}
\begin{proof}
We have
\Bea
m_{n}(x,{\bf e})=\int_{\XX}m_1(x,{\bf e}, dz)m_{n-1}(z,T{\bf e}) 
 &\leq & M({\bf e})\int_{\XX} m_1(y,{\bf e}, dz)m_{n-1}(z,T{\bf e})\\
& \leq & M({\bf e})m_{n}(y,{\bf e}).
\Eea
It yields  the first part of the lemma.
We then write
\Bea
Q_n(x,{\bf e}, A)&=&\int_{\XX}\frac{m_1(x,{\bf e},dz)}{m_n(x,{\bf e})} m_{n-1}(z,T{\bf e}, A)\nonumber \\
&&\leq M({\bf e})^2\int_{\XX}\frac{m_1(y,{\bf e},dz)}{m_n(y,{\bf e})} m_{n-1}(z,T{\bf e}, A)\leq M({\bf e})^2Q_n(y,{\bf e}, A) \label{DQ}
\Eea
to get the second part of the lemma.
\end{proof}
\begin{proof}[Proof of Corollary \ref{D}]  Using the branching property in generation $i$ and the first part of Lemma
\ref{lem}, we have for all $x,y \in \XX$,
$$m_{i+k}(x,{\bf e}) \geq m_{i}(x, {\bf e}) M(T^{i}e)^{-1} m_{k}(y,T^{i}{\bf e}).$$
Then, letting $y=x_0,x_1$, 
 $$ V_i(x_0,x_1)=\sup_{ k\geq 0 } \frac{m_{k}(x_0, T^{i}{\bf e})m_{k}(x_1, T^{i}{\bf e})}{m_{k+i}(x,{\bf e})^2} \leq  \frac{M(T^i{\bf e}) ^{2}}{m_{i}(x,{\bf e})^2} .$$
Moreover $m_i(x,{\bf e}) \geq m_{i-1}(x,{\bf e})M(T^{i-1}{\bf e})m(x,T^{i-1}{\bf e})$ and
\Bea
 \sum_{i \geq 1} \E_{{\bf e},\delta_x}\left( \sum_{\vert w\vert =i-1} \sum_{ \vert wa\vert=\vert wb\vert=i} V_i(X(wa),X(wb))\right)
&\leq &  \sum_{i \geq 1} \E_{{\bf e},\delta_x}(Z_{i-1}(\XX))\sigma(T^{i-1}{\bf e}) \frac{ M(T^i{\bf e}) ^{2}}{m_{i}(x,{\bf e})^2}\\
&\leq  &\sum_{i \geq 1}\sigma(T^{i-1}{\bf e}) \frac{M(T^{i-1}{\bf e})  M(T^i{\bf e}) ^{2}}{m(x,T^{i-1}{\bf e})m_{i}(x,{\bf e})}.
\Eea
So the first part of  (\ref{coeffD}) ensures that (\ref{Hyp1as}) is fulfilled and  Lemma \ref{ctrL2} ensures  that $Z_n(\XX)/m_n(x,{\bf e})$ is bounded in $L^2_{{\bf e}, \delta_x}$.
The first part of the Theorem is proved and we prove now the a.s. convergence.

Using the second part of Lemma \ref{lem}, we first get the geometric ergodicity of $Q_{i,n}$. Indeed, a simple induction leads to  
$$\vert Q_{i,n}(\lambda, {\bf e}, f)- Q_{i,n}(\mu, {\bf e}, f) \vert \leq \parallel f\parallel_{\infty}\prod_{j=i}^{n-1} (1-1/M(T^j{\bf e})^2)$$
and  the second part of (\ref{coeffD}) ensures that (\ref{Hyp2as}) hold.
Then Theorem \ref{coras} yields the expected a.s. convergence.
\end{proof}

\subsubsection{Under Lyapounov type  conditions}
One can relax  the conditions of Section \ref{Doeb} and still get 
the geometric ergodicity of   $Q_{i,n}$ via  Harris ergodic theorems. A well known method  relies  on the use of Lyapounov function outside a compact set, see e.g. \cite{MT, HM}. The results can be easily adapted to the non-homogeneous setting required here to get sufficient conditions for (\ref{strerg}), Assumption \ref{ergo}(b) or (\ref{Hyp2as}). \\
But one also need to control the underlying genealogy to get a law of large numbers. The additional work required to apply finely Theorem \ref{coras} seems beyond  the scope of this paper and 
could be a   future work. Let us note that  Theorem \ref{brchthm} can be more easily applied and yields to first interesting results. \\


\subsubsection{Comments on multitype branching processes}
When the state space $\XX$ is finite, the process $Z$ is a multitype branching process and much finer results can be obtained.
In particular, the limit behavior of  $Z_n/m_n(x,{\bf e}, \XX)$ is known, see e.g. \cite{KLPP} in fixed environment and \cite{Cohn} in random environment.
Let us still mention that we provide in Lemma \ref{manytoone}  a slightly different spine decomposition than \cite{KLPP}, without  projection with respect to the eigenvector associated to the mean operator. It might be of interest.
Finally, we note that  Corollary \ref{D} may be applied easily  to get new results  in varying environment.  \\

In the two next applications, we consider the case when the reproduction law does not depend
on the trait, so $Q_{i,n}$ just depends on $i$ and the auxiliary Markov chain with kernel
$Q_i$ is denoted by $Y$. We  focus in these two examples on the  
functions $f_n$ one can use to derive relevant results.

\subsubsection{Comments on branching random walks and random environment}
Branching random walks have been largely studied from the pioneering works of Biggins  (see e.g. \cite{biggins77}) and central limit theorems have been obtained to describe the repartition of the population 
for large times \cite{b}. \\
For branching random walks (possibly in varying environment in time and space), the auxiliary Markov chain $Y$ is a random walk 
(possibly in varying environment in time and space). To get law of large numbers for $Z_n$, one can then 
 check that some
 convergence in law 
$$(Y_n-a_n)/b_n\Rightarrow W$$
where the limit $W$ does not depend on the initial state $x\in \XX$. Then we can use Theorem \ref{LGNL2} with $f_n(x)=(x-a_n)/b_n$ to obtain the asymptotic proportion
of individuals whose trait $x$ satisfies $f_n(x)\in [a,b]$. It is given  by $\P(W\in [a,b])$ soon as $\P(W\in \{a,b\} )=0$.\\
Thus it can be used when the auxiliary process satisfies a central limit theorem. We refer to \cite{BanHua} Section 3.4 for some examples in the case when the reproduction law does not depend 
on the trait $x\in \XX$ and the environment is stationary ergodic in time. \\
One can actually directly derive some (rougher) law of large numbers dierctly from the speed of random walk (in environment), i.e. use
$Y_n/a_n \Rightarrow v$. As as example, we recall that in dimension $1$ the random walk in random environment $Y$ may be subalistic and $b_n=n^{\gamma}$ with $\gamma\in (0,1)$.\\
We finally  mention \cite{MN} when the offspring distribution is chosen in an i.i.d. manner for each time $n$ and location $x \in \ZZ$. \\

\subsubsection{Comments on Kimmel's cell infection model and non-ergodicity}
In the Kimmel's branching model \cite{ban} for cell division with parasite infection, the auxiliary Markov chain $Y_n$ is 
a Galton-Waston in (stationary ergodic) random environment. For example, in the case when no extinction is possible, i.e.
$\P_1(Y_1>0)=1$, under the usual integrability assumption we have
$$Y_n/\Pi_{i=0}^{n-1} m_i \stackrel{n\rightarrow \infty}{\longrightarrow} W \in (0,\infty) \qquad \text{a.s.}$$
where $m_i$ is the mean  number of offsprings for each parasite in generation $i$. We note that the distribution of
$W$ depends  on the inital value of $Y$.
But 
$$\log(Y_n)/n \stackrel{n\rightarrow \infty}{\longrightarrow} \E(\log m_0) \qquad \text{a.s.}$$
and the limit here does not depend on $Y_0$ anylonger. So we get the  ergodic property required to  use use Theorem \ref{LGNL2} with $f_n(x)=\log(x)/n$. We obtain that
 the proportion of cells in generation whose number of parasites is between $\exp([\E(\log m_0)-\epsilon]n)$ and $\exp([\E(\log m_0)+\epsilon]n)$ goes to $1$ in probability, for every $\epsilon>0$.
This yields some first new result on the infection propagation, which could be improved by additional work. \\
Soon as the number of parasites in a cell can be equal to zero, i.e. $\P_1(Y_1=0)>0$, ergodicity is failing
and some additional work is needed. Using monotonicity argument, one may still conclude, see \cite{ban} for an example. 


\section{Local densities and extremal particles.}

We deal now with local densities and the associated ancestral lineages. More precisely, we focus on the number of individuals whose trait belongs  to some set $A_n$ in generation $n$, when $n\rightarrow \infty$.\\
We have proved  the many-to-one formula
$$\E(Z_n(A_n))=m_n(x,{\bf e},\XX)Q_{0,n}(x,{\bf e},A_n)$$
in the previous section. We have then checked that the ergodicity of $Q_{0,n}$ ensures that $Z_n(A)/m_n(x,{\bf e},\XX)-Q_{0,n}(x,{\bf e},A)$ goes to zero under some conditions. \\
 Now we wish to compare the asymptotic behaviors of $Z_n(A_n)$ and $m_n(x,{\bf e},\XX)Q_{0,n}(x,{\bf e},A_n)$, when $Q_{0,n}(x,{\bf e},A_n)\rightarrow 0$  as $n\rightarrow \infty$. In particular, we are 
studying the links between the local densities $Z_n(A_n)$ for large times and  the large deviations events of $Q_{0,n}$, i.e. the asymptotic behavior of $Q_{0,n}(x,{\bf e}, A_n)$. \\
Such questions have been well studied for branching random walks from the pioneering work of Biggins  \cite{biggins77},  and
we refer to \cite{Rspa,HS} for related results and to \cite{Rouault, Shi} for reviews on the topic.
We also mention  \cite{CP, MN2} for the random environment framework and  \cite{DMS}
for large deviations Markov chain along tree with $n$ vertices. \\
The upper bound for such results comes directly from  Markov inequality and we are working on the lower bound. 
As usual, we could then   derive the rough asymptotic behavior of the extremal (minimal or maximal position)  individual. It covers  classical  results for branching random walks on the speed  of the extremal individual at the $\log$ scale. We provide some other examples motivated by cell's infection model, where the associated deviation strategy is more subtle. We mention also  that   $Z_n(A_n)$ may be  negligible compared
to $m_n(x,{\bf e},\XX)Q_{0,n}(x,{\bf e},A_n)$.\\

\begin{Def} For all $0\leq i\leq n$, $A,B\in \mathcal B_{\XX}$,  we define the measure
$$\mu_{i,n}(A, {\bf e}, B) [l, \infty):=\inf_{x\in A}\P_{\delta_x, T^i{\bf e}} (Z_{n-i}(B) \geq l)$$
\end{Def}
\noindent
We note $\bar{\mu}$ the mean of $\mu$ and $\hat{\mu}$ the variance of $\mu/\bar{\mu}$, so that
$$\bar{\mu}_{i,n}(A, {\bf e}, B)=\sum_{l\geq 1} \mu_{i,n}(A, {\bf e}, B) [l, \infty), \qquad \hat{\mu}_{i,n}(A, {\bf e}, B)=\frac{\sum_{l\geq 1} l^2\mu_{i,n}(A, {\bf e}, B) \{l\}}{\bar{\mu}_{i,n}(A, {\bf e}, B)^2} - 1.$$

We start by coupling our process in the first stages by a particular branching process in varying environment to  use both the convergence  of the associated martingale during these first steps $(i=0,\ldots,\phi(n))$ and a law of large number argument on the remaining time ($i=0,\ldots, \psi(n)$).
\begin{Lem}\label{klem} 
Let ${\bf e}$ and $x\in \XX$.  We assume that there exists a  non-decreasing sequence $k_i$ of integers  and a sequence   $B_i$ of subsets of  $\XX$ such that

$$x\in B_0, \qquad  \liminf_{i\rightarrow \infty} \bar{\mu}_{k_i,k_{i+1}}(B_i,{\bf e},B_{i+1})>1, \qquad
 \sum_{i\geq 0}  \frac{\hat{\mu}_{k_i,k_{i+1}}(B_i,{\bf e},B_{i+1})}{\Pi_{i=0}^{n-1}\bar{\mu}_{k_i,k_{i+1}}(B_i,{\bf e},B_{i+1})}<\infty.$$
Then there exists an event $A$ whose probability is positive, such that for any non-decreasing sequence of integers  $\phi_n,\psi_n$ and   $k_{i,n}$ 
and any  sequence $B_{i,n}$  of subsets of $\XX$ which satisfy
$$k_{0,n}=k_{\phi_n},  \qquad \phi_n\rightarrow \infty, \qquad \sup_n \sum_{i=0}^{\psi_n-1}  \frac{\hat{\mu}_{k_{i,n},k_{i+1,n}}(B_{i,n},{\bf e},B_{i+1,n})}{\Pi_{i=0}^{n-1}\bar{\mu}_{k_{i,n},k_{i+1,n}}(B_{i,n},{\bf e},B_{i+1,n})}<\infty,$$
we have
$$\left\{\liminf_{n\rightarrow \infty} \frac{Z_{n}(B_{n,\psi_n})}{P_n}>0\right\} \supset A,$$
where 
$$P_n:= \prod_{i=0}^{\phi_n-1}\bar{\mu}_{k_i,k_{i+1}}(B_i,{\bf e},B_{i+1}). 
\prod_{i=0}^{\psi_n-1} \bar{\mu}_{k_{i,n},k_{i+1,n}}(B_{i,n},{\bf e},B_{i+1,n}) $$
\end{Lem}

\begin{proof}
We use a  coupling of the branching Markov chain $Z$ with a supercritical Branching Process in Varying Environment (BPVE). Roughly speaking,  it is obtained by 
selecting the individuals whose lineage lives in the tube $(B_i :  i \leq \phi(n),B_{j,n} : j\leq \psi_n)$. More precisely, we consider the subpopulation of $Z$ constructed recursively by keeping the descendance of the population  in generation $k_i$ whose trait belongs to $B_i$ for $i\leq \phi_n$ and then  belongs to $B_{j,n}$ for $j\leq \psi_n$.
The size  of the population  obtained by this construction in generation $k_i$ is a.s. larger than a branching process $N_i$ whose reproduction law in generation $i$ is 
$\mu_i:=\mu_{k_i,k_{i+1}}(B_i,{\bf e},B_{i+1})$. Similarly, the size of the population in generation $k_{\psi_n,n}$ is larger than a branching process in varying environment, 
with initial value equal to $N_{\phi_n}$ and
successive reproduction laws $\mu_{j,n}:=\mu_{k_{j,n},k_{i+1,n}}(B_{j,n},{\bf e},B_{j+1,n})$ for $j=0,\cdots, \psi_n$.
Thus, 
$$Z_n(B_{n,\psi(n)})\geq \sum_{j=1}^{N_{\phi_n}} U_{j,n},$$
where $U_{j,n}$ is distributed as a BPVE in generation $\psi_n$, denoted  by $U_n$, whose successive reproduction laws are 
$\mu_{k_{i,n},k_{i+1,n}}(B_{i,n},{\bf e},B_{i+1,n})$
for $i=0, \cdots, \psi_n-1$. Moreover $(U_{j,n} : j=0, \ldots,  \psi_n$) are independent  by branching property.  \\
Using the assumption
 $$\sum_{i\geq 0} \frac{Var(\mu_i/\bar{\mu_i})}{\Pi_{j=0}^{i-1}\bar{\mu_i}}<\infty,$$
we know by  orthogonality that the martingale
$$\frac{N_i}{\prod_{j=0}^{i-1} \bar \mu_j}$$
converges in $L^2_{{\bf e}}$ and has a finite positive limit $W$ on the survival event $A:=\{\forall n\geq 0 :  N_n>0\}$.
Recalling that $\liminf_{i\rightarrow \infty} \bar{\mu}_{i} >1$ by assumption, $A$ has positive probability and conditionally on that event, we have a.s. 
$$\liminf_{i\rightarrow \infty}\frac{N_{i+1}}{N_i}>1.$$

Similarly $U_{j,n}/\E(U_{j,n})$ is bounded by the moment assumptions and
 $$X_{j,n}\stackrel{d}{=} \frac{U_{j,n}-\E(U_{n})}{\E(U_{n})}$$
are independent random variables, which are   independent of $(N_i :  i=0,\ldots, \phi_n)$ and   bounded in $L^2_{{\bf e}}$. 
Thanks to  Lemma  \ref{LemmLGN},
$$ \frac{1}{N_{k_{\phi_n}}} \sum_{j=1}^{ N_{k_{\phi(n)}} } X_{j,n}$$
goes to $0$ as $n\rightarrow \infty$ a.s. on the event $A$.
Then
$$Z_n(B_{n,\psi(n)})\geq N_{k_{\phi_n}}\E(U_n). \left[1 +\epsilon_n\right]$$
where $\epsilon_n\rightarrow 0$. Finally, we use
$$   \E(N_{k_{\phi_n}})=\prod_{i=0}^{\phi_n-1}\bar{\mu}_{k_i,k_{i+1}}(B_i,{\bf e},B_{i+1}), \quad \E(U_n)= \prod_{i=0}^{\psi_n-1} \bar{\mu}_{k_{i,n},k_{i+1,n}}(B_{i,n},{\bf e},B_{i+1,n}) $$
to get
$$
 \liminf_{n\rightarrow \infty} \frac{Z_{n}(B_{n,\psi_n})}{P_n} \geq \liminf_{n\rightarrow \infty} \frac{N_{\phi_n} }{\E(N_{\phi_n} )}(1+\epsilon_n)  \geq W.
$$
Recalling  that $A=\{W>0\}$ and it has positive probability ends up the proof.
\end{proof}

\subsection{Monotone Branching Markov chain}
Our aim is to see the local densities in terms of the large deviations of the auxiliary process and the way this large deviation event is achieved.
First, let us derive from the previous lemma  the number of individuals in   $A_n=[a_n,\infty):=\{ x \in \XX : x \geq a_n\}$ in the monotone case, which yields the applications
for the cell models and branching random walks which initially motivated these questions. \\
Thus, by now, we assume that $\XX$ is totally ordered by $\leq$ and the branching Markov chain  satisfies the following condition.
\begin{Ass}[Monotonicity] 
\label{monot}
For all $x\leq y$, ${\bf e} \in E$ and  $a\in \XX$, we have
$$\P_{\delta_x, \bf e} (Z_1([a,\infty)) \geq l)\leq \P_{\delta_y, \bf e} (Z_1 ([a,\infty))\geq l) \qquad (l\geq 0).$$
\end{Ass}

\begin{Ass}[Mean growth rate] \label{MG} Let $\rho>0$ such that
$$\lim_{n\rightarrow \infty} \frac{1}{n} \log m_{n}(x,{\bf e}, [a_n,\infty)) =\rho.$$
We also assume that  there  exist $p\geq 1$ and $b_i\in \XX$ such that $x\geq b_0$ and
$$
\liminf_{i\rightarrow \infty}  m_{p}(b_{i}, T^{ip}{\bf e}, [b_{i+1}, \infty)) >1.
$$
Finally,  for every $\epsilon>0$, there exist $q=q(\epsilon)$,   $\phi(n)\rightarrow \infty$ 
and $(b_{j,n} : j,n\geq 0)$ such that
$$ \liminf_{n\rightarrow \infty}\frac{1}{n} \sum_{ j < (n-\phi(n)p)/q} \log m_{q}(b_{j,n},T^{i\phi(n)+jq}{\bf e}, [b_{j+1,n}, \infty))\geq \rho-\epsilon.$$
\end{Ass}
The values $(b_{i} :  i\leq \phi(n) , b_{j,n} :  j\leq \psi(n))$ correspond to the (lower) curve   which yields the trait 
of the subpopulation which realizes the main contribution to the  population size $Z_n([a_n,\infty))$ in generation $n$. 
This curve is a (straight) line for   branching random walk or for $a_n=1$ in the Kimmel's branching model \cite{ban}, see below. But this curve is not straight  for the other quantities of interest
in  Kimmel's branching model, such as the large deviations associated to $a_n \rightarrow \infty$. Other motivating examples when 
 the curve are not a straight are given by large deviation events which are realized in one  step of the process. It can be  the case  for 
 random walks with heavy tails or  autoregressive processes.

\begin{Thm}  \label{ldmb} Let ${\bf e} \in E$ and $x\in \XX$. Under the  Assumptions  \ref{monot}, \ref{MG}   and 
$$\sup\big\{\E(N(z, T^k{\bf e})^2) : z\in \XX, k \geq 0\big\}<\infty,$$ then 
$$\P_{{\bf e}, \delta_x}\left(\frac{1}{n}\log Z_n([a_n,\infty))\stackrel{n\rightarrow\infty}{\longrightarrow} \rho \right)>0.$$
\end{Thm}
The upper bound of the convergence above is actually a.s. Letting the initial population go to infinity in this statement allows to get the convergence  a.s. by branching property. 
Getting the result a.s. on the survival event seems to require additional assumptions. \\
The uniform bound on the second moment assumption can be relaxed (see the proof), in particular the bound can depend on the environment to capture some branching models in
 random environment.
\begin{proof}
As for branching random walks, the upper bound comes directly from Markov inequality. For every $\eta >0$,
\Bea 
\P_{x,{\bf e}} ( Z_n([a_n,\infty)) \geq   \exp((\rho+\eta)n)) 
& \leq &\exp(-(\rho+\eta)n) m_n(x,{\bf e}, [a_n,\infty)), 
\Eea
so that  the first part of the Assumption \ref{MG}  ensures that the right hand side is summable. Then
Borel-Cantelli lemma yields the a.s. upper bound. 

The lower bound comes from the previous  Lemma with 
$$k_i=ip,\quad k_{n,j}=\phi_np+jq, \qquad B_j=[b_j,\infty), \quad B_{j,n}=[b_{j,n},\infty), \quad \psi_n=[(n-\phi_np)/q],$$
where $i=0,\ldots,\phi_n$, $j=0,\ldots, \psi_n$ and $[x]$ is the smallest integer larger or equal to $x$.
By Assumption \ref{monot} (monotonicity),
$$\mu_{k_j,k_{j+1}}(B_j,{\bf e}, B_{j+1})(.):=\P_{\delta_{b_j}, T^{k_j}{\bf e }}\big(Z_{k_{j+1}-k_j}([b_{j+1}, \infty)) = .\big)  $$
and the definition of $\mu_{k_{j,n},k_{j+1,n}}(B_{j,n},{\bf e}, B_{j+1,n})(.)$ is analogous. So 
$$\bar{\mu}_{k_j,k_{j+1}}(B_j,{\bf e}, B_{j+1})=m_{k_{j+1}-k_j}(b_j, T^{k_j}{\bf e },[b_{j+1}, \infty)) = m_{p}(b_{j}, T^{jp}{\bf e}, [b_{j+1}, \infty)). $$
and the analogous identity hold for $\bar{\mu}_{k_{j,n},k_{j+1,n}}(B_{j,n},{\bf e}, B_{j+1,n})$
By  Assumption \ref{MG}, we have for $\epsilon \in (0,\rho)$,
$$ \liminf_{j\rightarrow \infty} \bar{\mu}_{k_j,k_{j+1}}(B_j,{\bf e}, B_{j+1})>1, \ \ \liminf_{n\rightarrow \infty} \frac{1}{n} \log(\Pi_{j=0}^{\psi_n-1}\bar{\mu}_{k_{j,n},k_{j+1,n}}(B_{j,n},{\bf e}, B_{j+1,n}))\geq \rho-\epsilon>0.$$
Recalling that  $\sup\{\E(N(x, T^k{\bf e})^2) : x\in \XX, k \geq 0\}<\infty$ is assumed, we get 
$$\sum_{i\geq 0}  \frac{\hat{\mu}_{k_i,k_{i+1}}(B_i,{\bf e},B_{i+1})}{\Pi_{j=0}^{i-1}\bar{\mu}_{k_j,k_{j+1}}(B_i,{\bf e},B_{j+1})}<\infty; \quad 
\sup_n \sum_{i=0}^{\psi_n-1}  \frac{\hat{\mu}_{k_{i,n},k_{i+1,n}}(B_{i,n},{\bf e},B_{i+1,n})}{\Pi_{j=0}^{i-1}\bar{\mu}_{k_{j,n},k_{j+1,n}}(B_{j,n},{\bf e},B_{j+1,n})}<\infty.$$
Thus, we can apply  Lemma \ref{klem} and get
$$A\subset \left\{ \liminf_{n\rightarrow \infty}\frac{Z_{n}([a_n,\infty))}{\prod_{i=1}^{\psi_n} \bar{\mu}_{k_{i,n},k_{i+1,n}}(B_{i,n},{\bf e},B_{i+1,n}) } >0 \right\}
\subset \left\{\liminf_{n\rightarrow \infty} \frac{1}{n}\log Z_{n}([a_n,\infty)) \geq \rho-\epsilon \right\}$$
Noting that $A$ is fixed when $\epsilon\rightarrow 0$ and $\P(A)>0$ ends up the proof.
\end{proof}

As expected, we can now precise the asymptotic behavior of the extremal individuals.
 If  $a_n(x)$ satisfies the  assumptions of Theorem \ref{ldmb} with  some rate $\rho(x)$,
then, for every $x$ such that $\rho(x)>\log m$, 
$$\limsup_{n\rightarrow \infty} \frac{\max \{ X(u) : \vert u \vert=n \}}{a_n(x)} \leq 1 \qquad \P_{\delta_x, \bf e} \text{ a.s.}$$
and for every $x$ such that $\rho(x)<\log m$,
$$\liminf_{n\rightarrow \infty}\frac{\max \{ X(u) : \vert u \vert=n\}}{a_n(x)}  \geq 1 \qquad \P_{\delta_x, \bf e} \text{ a.s.}$$
  on some event whose probability is positive. \\
The proof is  standard.
 The first part comes directly from  Borel-Cantelli Lemma, recalling that
$$\P_{\delta_x, \bf e}(\max \{ X(u) : \vert u \vert=n \}\geq a_n(x))\leq \E_{\delta_x, \bf e}(Z_n([a_n(x),\infty)))=m(x,{\bf e}, [a_n(x),\infty))$$
decreases exponentially with rate $\rho-\rho(x)$.
The second part comes from the Theorem \ref{ldmb} which ensures that there are many particles beyond $a_n(x)$.

\subsection{Monotone Markov chain indexed by branching trees}

Let us specify in a simpler framework the results above, more precisely the  link between the local densities and the large deviations of the auxiliary
chain. We assume here that the reproduction law  does not depend on the trait of the individual, so that
\be
\label{ind}
N({\bf e}):=N(x,{\bf e}), \quad m({\bf e}):= m(x,{\bf e}), \quad 
 \qquad m_n({\bf e}):=m_n(x,{\bf e},\XX)=\prod_{i=0}^{n-1} m(T^{i}{\bf e}).
 \ee
As  above, we require the monotonocity of the trait distribution :
assume :
\begin{Ass}[Monotonicity of $P$] 
\label{mon}
For all $x\leq y$, ${\bf e} \in E$ and  $a\in \XX$, we have
$$P(x,{\bf e}, [a, \infty)) \leq P (y, {\bf e}, [a,\infty)).$$
\end{Ass}

We assume also that the large deviations of $Q_{i,n}$  beyond $a_n$  occur with rate $\alpha>0$ and that the beginning of the associated trajectory is supercritical, i.e.
\begin{Ass}[Large deviations of the auxiliary process $Q$] \label{AssLD} 
There exists $\alpha \geq 0$ such that
$$\lim_{n\rightarrow \infty} \frac{1}{n} \log Q_{0,n}(x,{\bf e}, [a_n,\infty)) =-\alpha$$
Moreover, we assume that there exist $p\geq 1$ and $b_i\in \XX$ such that
$$\liminf_{i\rightarrow \infty}  m_p(T^{ip}{\bf e})Q_{p}(b_{i}, T^{ip}{\bf e}, [b_{i+1}, \infty)) >1$$
and that for every $\epsilon>0$, there exist $q=q(\epsilon)$,   $\phi(n)\rightarrow \infty$ and $(b_{j,n} : j,n\geq 0)$ such that
$$ \liminf_{n\rightarrow \infty}\frac{1}{n} \sum_{ j < (n-\phi(n)p)/q} \log Q_{q}(b_{j,n},T^{i\phi(n)+jq}{\bf e}, [b_{j+1,n}, \infty))\geq -\alpha-\epsilon.$$
\end{Ass}

These assumptions are satisfied for the applications we have in mind.
For an example of   large deviations following Assumption \ref{AssLD}, a sufficient condition is  $P_{a_n}(Y_n \geq a_n+b_n)\sim \P_0(Y_n \geq b_n)$. The  trajectory associated to the large deviation event is then straight  and we can choose
$k_{i,n}=k_i$. It holds for random walks and more generally for random walks in random environment under general moment assumptions. 


\begin{Cor} \label{corLocal} Let ${\bf e} \in E$ and $x\in \XX$.
If (\ref{ind}),
$\sup\{\E(N(T^k{\bf e})^2) :  k \geq 0\}<\infty$   and Assumptions  \ref{mon} and \ref{AssLD} hold,  we have  
$$\P_{{\bf e}, \delta_x}\left(\frac{1}{n}\log \left(Z_n([a_n,\infty))/m_n({\bf e})\right)\stackrel{n\rightarrow\infty}{\longrightarrow} -\alpha\right)>0.$$ 
\end{Cor}
\noindent As expected, the large deviation of the auxiliary Markov chain quantifies the lost of growth $\alpha$ of the size of the population beyond $a_n$, $Z_n([a_n,\infty))$, compared to the whole growth of the population given by 
$m_n({\bf e})$. In the case of fixed environment, let us mention a related work on critical branching Markov chain \cite{ganmull}, where the recurrence property is investigated.

\subsection{Applications}

We first give some details on a motivating example for which straight and non straight curve for large deviations appear. We then give some first comments
on a new possible challenging questions.

\subsubsection{Kimmel's branching model}
We refer to \cite{ban} for a complete description of the model and the motivations.
The population of individuals is  a binary tree of cells and the trait is the number of parasites of the cell. The auxiliary Markov process $Y$ is then
a branching process in random environment.
Monotonicity (Assumption \ref{mon}) is a direct consequence of the branching property of $Y$. Tackling the local densities   and the trait of extremal individuals thanks to the previous Corollary  (only)
requires   to check Assumption \ref{AssLD}. \\

One of the motivating question in \cite{ban} is to count the number of infected cells in the subcritical case, which means that $Y$ is a.s. absorbed in finite time. Three regimes appear in the subcritical case \cite{GKV} and in particular in the weak subcritical case 
$$\P(Y_n>0) \sim cn^{-3/2}\gamma^n,$$
where $\gamma<\E_1(Y_1)$.
The mean number of infected cells is equal to $2^n\P(Y_n>0)$ and obtaining a.s. results was left open is this regime.
Corollary \ref{corLocal}
ensures that  if $2\gamma>0$, the number $N_n^*$  of infected cells in generation $n$ satisfies
$$\frac{1}{n}\log(N_n^*)\stackrel{n\rightarrow \infty}{\longrightarrow} \log(2\gamma) \qquad \text{a.s.}$$
on the event when the whole population of parasites survives. Indeed, this result is applied for $p$ large enough such that $2^p\P_1(Y_p>0)>1$, $a_n=[1,\infty]$,
$b_i=1, b_{j,n}=1$, $\phi(n)=o(n)$ and $q$ is chosen such that
$$\log \P_1(Y_q>0)\geq q\log(\gamma ) -\epsilon.$$

Second, when counting the number of cells infected less than the typical cell in the supercritical regime, the problem is now 
linked to the lower large deviation of branching processes in random environment $Y_n$, i.e. to 
$$\P(1 \leq Y_n \leq \exp(n\theta)), \qquad \text{where } \ \theta<\E(\log m(\EE)))$$
and the way this  large deviation event is realized. We refer to \cite{BB12} for the results. Here again Corollary
\ref{corLocal} allows to determine the a.s.  behavior of the number of cells whose  number of parasites is between $1$ and $\exp(\theta n)$. It is worth noting that for this question the associated trajectory is not straight and  $k_{i,n}$
depends on $n$.

\subsubsection{Comments on branching random walks and random environment}
We can recover here the classical result on the asymptotic behavior of 
$$\frac{1}{n}\log Z_n[an,\infty)$$
for a branching random walk with random increment $X$. It converges a.s. to 
$$\log(m)-\Lambda(a)$$
soon as $a\geq  \E(X)$ and $\log(m)>\Lambda(a)$,
where $\Lambda$ is the rate function associated to the random walk $S=\sum_{i=0}^{n-1} X_i$, see e.g. \cite{Rouault,Shi}.

One can extend this result to offsprings distribution in time varying environment and random walks in varying environment using the last Corollary and
 large deviations of random walks in varying environment. Here $b_i=aip$, $b_{j,n}=ajp+b_{\phi(n)}$, $\phi(n)=o(n)$.
We refer in particular to \cite{Zeit} for results on quenched and annealed  large deviations of random walk in random environment.\\
The uniform bound of the second moment of the reproduction in Theorem \ref{ldmb} can be relaxed and depend on the environment. Similarly, 
Assumption  \ref{MG}  can be extended to 
$$
\liminf_{i\rightarrow \infty}  m_{p}(b_{i}, T^{ip_i({\bf e})}{\bf e}, [b_{i+1}, \infty)) >1.
$$
Thus, using still Lemma \ref{klem}, one can get   an analogous result  in stationary random environment.

\subsubsection{Perspectives}
A main motivation  for this work is the control of local densities in cell division models for aging \cite{guyon, DelMar},  for damages \cite{evanssteinsaltz} or infection such as Kimmel's branching model already mentioned. An other motivation comes from spatial models in ecology with time and/or space inhomogeneity.
We aim at investigating further these questions and determine the behavior of extremal particles in these models, which seem to show  different large deviation's curves.


$\newline$

\textbf{Acknowledgement.} 
This work  was partially  funded by  Chair Modelisation Mathematique et Biodiversite VEOLIA-Ecole Polytechnique-MNHN-F.X., the professorial chair Jean Marjoulet, the project MANEGE `Mod\`eles
Al\'eatoires en \'Ecologie, G\'en\'etique et \'Evolution'
09-BLAN-0215 of ANR (French national research agency).


\begin{thebibliography}{12}

\bibitem[{AS10}]{AS} E. Aïdékon, Z. Shi (2010).  Weak convergence for the minimal position in a branching random walk: a simple proof. \emph{Periodica Mathematica Hungarica} 61 (2010) 43-54.

\bibitem[AK98a]{AKh} K. Athreya, H.J. Kang  (1998). Some limit theorems for positive recurrent branching Markov chains I. \emph{Adv. Appl. Prob. }30(3). 693-710.

\bibitem[AK98b]{AKh2} K. Athreya, H.J. Kang  (1998). Some limit theorems for positive recurrent branching Markov chains II. \emph{Adv. Appl. Prob. }30(3). 711-722.


\bibitem[AK71]{AK2} K. B. Athreya, S. Karlin (1971).  On branching processes with random environments  II : Limit theorems. \emph{Ann. Math. Stat. }42. 1843-1858.
\bibitem[{A00}]{AthreyaB} K. B.  Athreya (2000). Change of measure of Markov chains and the $L\log L$ theorem for branching processes.  \emph{Bernoulli} Vol. 6, No 2, 323-338.

\bibitem[{B08}]{ban}
V.~Bansaye.
\newblock Proliferating parasites in dividing cells : Kimmel's branching model
  revisited.
\newblock \emph{Annals of Applied Probability, Vol 18, Number 3, 967-996}
  (2008).


 \bibitem[BB12]{BB12} V. Bansaye, C. Boeinghoff.  Lower large deviations for supercritical branching processes in random environment (2012). To appear in the \emph{Proc.  of Steklov Institute of Mathematics}.

\bibitem[{BDMT11}]{BDMT}
V.~Bansaye, J.-F. Delmas, L.~Marsalle and V.C. Tran (2011).
\newblock Limit theorems for {M}arkov processes indexed by continuous time
  {G}alton-{W}atson trees.
\emph{Ann.  Appl.  Probab.}  Vol. 21, No. 6, 2263-2314.

\bibitem[{BL12}]{BL} V. Bansaye, A. Lambert (2012).
New approaches of source-sink metapopulations decoupling the roles of demography and dispersal. To appear in \emph{Theor. Pop. Biology}

\bibitem[{BH13}]{BanHua} V. Bansaye, C. Huang (2013). Weak law of large numbers for some Markov chains along non homogeneous genealogies. \emph{Preprint} avialable via Arxiv.

\bibitem[B77]{biggins77} J.D. Biggins  (1977). Martingale convergence in the branching random walk.\emph{ J. Appl. Probab.} \textbf{14}, 25-37.

\bibitem[B90]{b} J. D. Biggins  (1990). The central limit theorem for the
supercritical branching random walk and related results.  \emph{Stoch. Proc. Appl.} \textbf{34}, 255-274.

\bibitem[{C11}]{BC}  B. Cloez (2011). Limit theorems for some branching measure-valued processes. \emph{Avialable via}  http://arxiv.org/abs/1106.0660.

\bibitem[{CRW91}]{CRW} B. Chauvin, A. Rouault, A. Wakolbinger (1991). Growing conditioned trees. \emph{Stochastic Processes and their Applications} \textbf{39}, 117--130.

\bibitem[{C89}]{Cohn} H. Cohn (1989). On the growth of the supercritical multitype branching processes in random environment. \emph{Ann. Probab.} Vol. 17, No 3. 1118-1123.

\bibitem[{CP07a}]{CoPo} F. Comets, S. Popov  (2007). Shape and local growth for multidimensional
branching random walks in random environment. \emph{ALEA} \textbf{3}, 273-299.

  \bibitem[CP07b]{CP} F. Comets, S. Popov  (2007). Shape and local growth for multidimensional
branching random walks in random environment. \emph{ALEA} \textbf{3}, 273-299.

\bibitem[CY11]{CY}  F. Comets, N. Yoshida  (2011).  Branching random walks in space-time random environment: survival
probability, global and local growth rates.
\emph{Journal of Theor.  Probab.} 

\bibitem[DZ98]{DZ} A. Dembo, O. Zeitouni (1998). \emph{Large deviations techniques and applications}. Applications of Mathematics (New York) 38 (Second edition ed.)

\bibitem[DMS05]{DMS} A. Dembo, P.  M\"orters, S. Sheffield (2005). Large deviations of Markov chains indexed by random trees. \emph{Ann. Inst. H. Poincaré Probab. Statist.}  41, no. 6, 971?996. 

\bibitem[{DM10}]{DelMar}
J.-F. Delmas, L.~Marsalle (2010).
\newblock Detection of cellular aging in a {G}alton-{W}atson process.
\newblock \emph{Stochastic Processes and their Applications} \textbf{120},
  2495--2519.


\bibitem[{E07}]{Engll}
J.~Engländer (2007).
\newblock Branching diffusions, superdiffusions and random media.
\newblock \emph{Prob. Surveys.} \textbf{4}, 303-364.

\bibitem[{EK86}]{ethierkurtz}
S.N. Ethier, T.G. Kurtz (1986).
\newblock \emph{Markov Processus, Characterization and Convergence}.
\newblock John Wiley \& Sons, New York.

\bibitem[{ES07}]{evanssteinsaltz}
S.N. Evans,  D.~Steinsaltz (2007).
\newblock Damage segregation at fissioning may increase growth rates: A
  superprocess model.
\newblock \emph{Theoretical Population Biology} \textbf{71}, 473-490.

\bibitem[{F71}]{fellerlivre}
W.~Feller (1971).
\newblock \emph{An introduction to probability theory and its applications},
  volume 1 and 2.
\newblock Wiley.


\bibitem[{FK60}]{Kesten}  H. Furstenberg,   H. Kesten (1960). Products of random matrices. \emph{The Annals of Mathematical Statistics}, 31(2):457-469.



\bibitem[GM05]{ganmull} N. Gantert, S. M\"uller .The critical Branching Markov Chain is transient. \emph{Arxiv} http://arxiv.org/abs/math/0510556v1.


\bibitem[GMPV10]{gant} N. Gantert, S. M\"uller, S. Popov, M. Vachkovskaia  (2010). Survival of branching random walks in random environment.	\emph{Journal of Theor. Probab.}, \textbf{ 23}, 1002-1014.

\bibitem[GKV03]{GKV} J. Geiger, G. Kersting, V. A.  Vatutin (2003).
Limit theorems for subcritical branching processes in random environment. 
\textsl{Ann. Inst. Henri  Poincar\'{e} (B). }\textbf{39}, pp. 593--620.

\bibitem[{G99}]{Geiger} J. Geiger  (1999). Elementary new proofs of classical limit theorems for
Galton-Watson processes. \emph{J. Appl. Prob. 36}, 301-309.

\bibitem[{GB03}]{GeBa} H.O. Georgii,   E. Baake (2003). Supercritical multitype branching processes: the ancestral types of typical individuals.
\emph{Adv. in Appl. Probab.}, 
Vol. 35,
No 4, 1090-1110. 

\bibitem[{GRW92}]{GRW} L. G. Gorostiza,  S. Roelly,  A. Wakolbinger (1992). Persistence of critical multitype particle and measure branching processes.
 \emph{Probability Theory and Related Fields}, Vol. 92, No {3}, 313-335.


\bibitem[{G07}]{guyon}
J.~Guyon.
\newblock Limit theorems for bifurcating {M}arkov chains. {A}pplication to the
  detection of cellular aging.
\newblock \emph{Ann.  Appl. Probab.} \textbf{17}, 1538--1569 (2007).


\bibitem[{HM08}]{HM}  M.  Hairer, J. C. Mattingly (2008).  Yet another look at Harris' ergodic theorem for Markov chains. \emph{Avialable via}  arXiv:0810.2777. 

\bibitem[HR12]{HR} S. C. Harris, M. I.  Roberts (2012). The many-to-few lemma and multiple spines. \emph{Avialable via} http://arxiv.org/abs/1106.4761.

\bibitem[{HR13}]{HRb} S. C. Harris, M. I. Roberts (2013). A strong law of large numbers for branching processes: almost sure spine events
\emph{Avialable via } http://arxiv.org/abs/1302.7199.


\bibitem[HL11]{huang2} C. Huang, Q. Liu  (2011).  Branching random walk with a random environment in time. \emph{Preprint.}

\bibitem[{HS09}]{HS} Y. Hu, Z. Shi (2009). Minimal position and critical martingale convergence in branching random walks, and directed polymers on disordered trees. \emph{Ann.  Probab.} 37  742-781.


\bibitem[{JN96}]{Jagers} P. Jagers, O. Nerman (1996). The asymptotic composition of supercritical multi-type branch-
ing populations. In \emph{S\'eminaire de Probabilit\'es}, XXX, volume 1626 of Lecture Notes in
Math., pages 40-54. Springer, Berlin, 1996.



\bibitem[{K74}]{Kaplan}  N. Kaplan  (1974).
Some Results about Multidimensional Branching Processes with Random Environments. \emph{Ann. Probab.},  Vol. 2, No. 3., 441--455.

\bibitem[{K72}]{Kurtz72} T. Kurtz (1972). Inequalities for law of large numbers. \emph{Ann.Math.Statist.} 43,1874-1883.


\bibitem[KLPP97]{KLPP}  T. Kurtz, R. Lyons, R. Pemantle, Y. Peres (1997). A conceptual
proof of the Kesten-Stigum theorem for multi-type branching processes. In
\emph{Classical and Modern Branching Processes}, ed. K. B. Athreya and P. Jagers. Springer, New York. 181-185.

\bibitem[LPP95]{LPP} R. Lyons,  R. Pemantle, Y. Peres (1995). Conceptual proofs of $L\log L$ criteria for mean behavior of branching processes. \emph{Ann.   Probab.}, Vol. 23, No 3,
 1125-1138. 
 
 \bibitem[MT09]{MT} S. Meyn L.  Tweedie (2009). \emph{Markov Chains and Stochastic Stability}. Broch\'e.

\bibitem[M13]{Mukh}   F. Mukhamedov (2013). Weak ergodicity of nonhomogeneous Markov chains on noncommutative L1-spaces. \emph{Banach J. Math. Anal.} Vol. 7, No. 2, 53 -73.

\bibitem[M08]{mull} S. M\"uller  (2008). A criterion for transience of multidimensional branching random walk in random environment. \emph{Electr. Jour. Probab.} \textbf{13}, 1189-1202.

\bibitem[N11]{MN}  M. Nakashima (2011). Almost sure central limit theorem for branching random walks in random environment. 
\emph{Ann. Appl. Probab.}, \textbf{21}(1), 351-373.

\bibitem[N12]{MN2} M. Nakashima (2013). Minimal Position of Branching Random Walks
in Random Environment. \emph{J. Theor Probab}.  26:1181-1217


\bibitem[NJ84]{NJ} O. Nerman,  P. Jagers (1984). The stable double infinite pedigree process of supercritical branching populations.  \emph{Z. Wahrsch. Verw. Gebiete}  65  ,  no. 3, 445-460.


\bibitem[{R93}]{Rspa} A.  Rouault (1993). Precise estimates of presence probabilities in the branching random walk. \emph{Stoch.  Process. Appl.} 44, no. 1, 27-39.

\bibitem[R00]{Rouault} A. Rouault (2000). Large deviations and branching processes. \emph{Proceedings of the 9th International Summer School on Probability Theory and Mathematical Statistics} (Sozopol, 1997). Pliska Stud. Math. Bulgar. 13, 15-38.

\bibitem[{S94}]{timo} T. Sepp\"al\"ainen  (1994). Large deviations for Markov chains with Random Transitions. \emph{Ann. Prob.} \textbf{22} (2), 713-748.

\bibitem[S08]{Shi} Z. Shi (2008).  \emph{Random walks and trees.} Lecture notes, Guanajuato, Mexico, November 3-7.



\bibitem[{T88}]{Tanny} D. Tanny (1988). A necessary and sufficient condition for a branching process in a random environment to grow like the product of its means.  \emph{Stoch.  Process. Appl.}  28,  no. 1, 123-139.

\bibitem[Y08]{yosh}  N. Yoshida (2008). Central limit theorem for random walk in random environment. \emph{Ann.  Appl. Probab.}  \textbf{18} (4), 1619-1635.

\bibitem[{Z04}]{Zeit} O. Zeitouni (2004). Random walks in random environment. \emph{XXXI Summer school in probability, St
Flour} (2001). Lecture notes in Math. 1837 (Springer) (2004), pp. 193?312.
\end{thebibliography}
\end{document}